\documentclass[aop]{imsart}
\pdfoutput=1
\RequirePackage{amsthm,amsmath,amsfonts,amssymb}
\RequirePackage[numbers]{natbib}
\RequirePackage[colorlinks,citecolor=blue,urlcolor=blue]{hyperref}
\RequirePackage{graphicx}

\usepackage[utf8]{inputenc}
\usepackage{amsfonts,amsmath,amssymb, amsthm, times}
\usepackage{mathrsfs}

\startlocaldefs
\theoremstyle{plain}

\newtheorem{theorem}{Theorem}[section]

\theoremstyle{remark}

\newtheorem{thm}{Theorem}[section]

\newtheorem{lem}[thm]{Lemma}
\newtheorem{prp}[thm]{Proposition}

\newtheorem{remark}{Remark}[section]

\theoremstyle{definition}
\newtheorem{defn}{Definition}[section]
\newcommand{\scr}[1]{\mathscr #1}

{\newtheorem{theorem*}{Theorem}
\newtheorem{lemma*}{Lemma}
\newtheorem{proposition*}{Proposition}
}

\numberwithin{equation}{section} \theoremstyle{remark}

\newcommand*{\email}[1]{\href{mailto:#1}{\nolinkurl{#1}} }

\def\Ee{\mathbb{E}}
\def\e{\varepsilon}
\def\1{{\bf 1}}
\def\N{{\mathbb N}}
\def\R{\mathbb R} \def\EE{\mathbb E}

\def\<{\langle} \def\>{\rangle}  
   
\def\d{\text{\rm{d}}}    
  
\def\f{\frac}
\def\beg{\begin}   
\def\Ric{\text{\rm{Ric}}}
\def\ric{\mathop {\rm ric}}

\def\e{\text{\rm{e}}}   
      \def\Ric{\text{\rm{Ric}}}
\def\cut{\text{\rm{cut}}}

\def\Pp{\mathbb{P}}

\def\div{\text{\rm div}}

\def\cut{\text{\rm cut}}
\def\supp{\text{\rm supp}}

\def\e{\varepsilon}
\def\w{\varpi}
\def\so{{\mathfrak {so}}}
\def\paral{/\kern-0.55ex/}
\def\parals_#1{/\kern-0.55ex/_{\!#1}}
\def\bparals_#1{\breve{/\kern-0.55ex/_{\!#1}}}
\def\ric{\text{\rm ric}}

\endlocaldefs

\begin{document}

\begin{frontmatter}
\title{Logarithmic heat kernel estimates without curvature restrictions}
\runtitle{A sample running head title}

\begin{aug}
\author[A]{\fnms{Xin}~\snm{Chen}\ead[label=e1]{chenxin217@sjtu.edu.cn}},
\author[B]{\fnms{Xue-Mei}~\snm{Li}\ead[label=e2]{xue-mei.li@epfl.ch}}
\and
\author[C]{\fnms{Bo}~\snm{Wu}\ead[label=e3]{wubo@fudan.edu.cn}}
\address[A]{School of Mathematical Sciences,
 Shanghai Jiao Tong University\printead[presep={,\ }]{e1}}

\address[B]{ Imperial College London and EPFL\printead[presep={,\ }]{e2}}

\address[C]{School of Mathematical Sciences, Fudan University\printead[presep={,\ }]{e3}}
\end{aug}

\begin{abstract}
The main results  of the article are short time estimates and asymptotic estimates for the first two order derivatives of the logarithmic
 heat kernel of a complete Riemannian manifold. We remove all  curvature restrictions  and also develop several techniques.

A basic tool developed here is intrinsic stochastic variations with prescribed second order covariant differentials, allowing to obtain a path integration representation  for the second order derivatives of the heat semigroup $P_t$ on a complete Riemannian manifold, again without any assumptions on the curvature. The novelty is the introduction of an $\epsilon^2$ term in the variation allowing  greater control. We also construct a family of cut-off stochastic processes adapted to an exhaustion by compact subsets with smooth boundaries, each  process is constructed path by path and differentiable in time, furthermore the differentials have locally uniformly bounded moments with respect to the Brownian motion measures, allowing to by-pass the lack of continuity of the exit time of the Brownian motions on its initial position.
\end{abstract}



\end{frontmatter}

\section{Introduction}

Let $(M,g)$ be an $n$-dimensional connected and
complete Riemannian  manifold endowed with the Levi-Civita connection $\nabla$.  Let $\Delta$  denote the  Laplace-Beltrami operator, and
 let $p(t,x,y)$ denote its heat kernel, by which we mean the minimal positive fundamental solution to the equation $\f \partial {\partial t} =\f 12 \Delta $.
The objective of this article is to provide estimates on the first and the second order gradients of $\log p(t,x,\cdot)$, without imposing any curvature conditions on $M$.
For a fixed $x\in M$, we  use the abbreviation $\log p$ for the logarithmic heat kernel
$\log p(t,x,\cdot)$ and use $\nabla \log p$ and $\nabla^2 \log p$ for its first and second  order derivatives respectively.

We  begin with explaining some of the motivations and potential applications.
Let  $o\in M$ be fixed, we denote
$$P_o(M):=\{\gamma\in C([0,1];M): \gamma(0)=o\}$$
the based path space over $M$.
Likewise, let $L_o(M)$ denote the based loop space over $M$,
$$L_o(M):=\{\gamma\in P_o(M):\ \ \gamma(0)=\gamma(1)=o\}.$$
A classical problem is to seek a suitable probability measure on $P_o(M)$ or $L_o(M)$, 
with which analysis on these infinite dimensional non-linear spaces can be made and understanding of the path spaces can be furthered.
If $M$ is compact or more generally with bounded geometry, a natural candidate for the probability measure on $L_o(M)$ is the probability distribution of the diffusion process with the infinitesimal operator
$$L:=\f 12\Delta + \nabla \log p(1-t,\cdot, o)$$
and the initial value $o$. This is  the Brownian bridge measure.
Since there is no analogue of a Lebesque  measure,  translation invariant, on $L_o(M)$,  the Brownian bridge measure is essentially the canonical measure to use. Indeed, for $M=\R^n$ the Brownian bridge measure is a Gaussian measure and it is quasi-invariant under translations of Cameron-Martin vectors.
To construct such a diffusion process, which is usually called  the Ornstein-Uhlenbeck process, we define a pre-Dirichlet form. This form will be called the Ornstein-Uhlenbeck (O-U) Dirichlet form.
To verify that the pre-Dirichlet form yields a Markov process, it is necessary to show it is closed -- a property following readily once we have  an integration by parts (IBP) formula. The key ingredient for such an IBP formula is suitable short time estimates on $\nabla \log p$ and $\nabla^2 \log p$.
We refer the reader to Aida \cite{A,A3}, Airault and Malliavin \cite{Airault-Malliavin}, Driver \cite{D2}, Hsu \cite{Hsu2}  and
Li \cite{Li2} for more detail.

Another interesting problem is to establish functional inequalities for
the O-U Dirichlet form. This includes the Poincar\'e inequality and logarithmic Sobolev inequality. They describe the long time behaviours of the associated diffusion process.  The logarithmic Sobolev inequality for Gaussian measures was  obtained by  Gross  in the celebrated paper \cite{Gro75}.
However,  this is not known to hold for loop space over a general manifold $M$. When $M$ was the hyperbolic space, Poincar\'e inequality  was shown to hold on $L_0(M)$ by the authors of the article \cite{CLW1} and Aida \cite{A4}. If $M$ was compact simply connected with strictly positive Ricci curvature, a weak Poincar\'e inequality with explicit rate function was also established by the authors of the article \cite{CLW2}.
It was shown in Gross \cite{Gross-91} that the Poincar\'e inequality for O-U Dirichlet form did not hold on $L_o(M)$ when $M$ was not simply connected. Soon after, Eberle
\cite{E1} constructed a simply connected compact manifold  for which the Poincar\'e inequality for O-U Dirichlet form did not hold on $L_o(M)$.
When the based manifold $M$ was compact, Aida \cite{A}, Eberle \cite{E}, Gong and Ma \cite{GM}, Gong, R\"ockner and Wu
\cite{Gong-Rockner-Wu} and Gross \cite{Gross-91} have obtained
weighted log-Sobolev inequalities or other different versions of modified log-Sobolev inequalities on $L_o(M)$.
 In all the
results mentioned above, the crucial ingredient was again the asymptotic estimates for $\nabla \log p$ and $\nabla^2\log p$.

We want to stress that all the results mentioned above have been established for the base manifold $M$ compact or  with some bounded geometry conditions,  since
the short time or asymptotic estimates for $\nabla \log p$ and $\nabla^2\log p$ were only known for manifolds with such restrictions.
Our immediate concern is to study the construction of diffusion processes and functional inequalities on
$L_o(M)$ without any bounded geometry conditions on $M$. We will obtain short time or asymptotic estimates for $\nabla \log p$ and $\nabla^2\log p$
in this paper. These estimates will be applied to study several problems on $L_o(M)$ in a forthcoming paper \cite{CLW}.

It is  intriguing that estimates for $\nabla \log p $ and  $\nabla^2 \log p$  are also main tools  for proving the continuous counterpart of Talagrand's conjecture for the hypercube $\Omega_n=\{-1, 1\}^n$ which we explain below. Let $\sigma^i$ denote  the configuration with the ith coordinate of $\sigma$ flipped and let $\sigma_i$ denote the i-th component of $\sigma\in \Omega_n$. Let $\mu_n\equiv 2^{-n}$ be the uniform measure on $\Omega_n$ which is reversible associated with the  generator
$Lf(\sigma):=\f 12 \sum_{i=1}^n (f(\sigma^i)-f(\sigma))$ where $\sigma\in\Omega_n$.
Setting
$T_sf(\sigma):=\int_{\Omega_n} f(\eta) \Pi_{i=1}^n (1+e^{-s}  \sigma_i \eta_i) d\mu_n(\eta)$, then
Talagrad's conjecture states that for any $s>0$ there exists a constant $c_s$ independent of the dimension $n$ such that
$\mu_n\Big(\Big\{\sigma: T_sf(\sigma)\ge t\Big\}\Big)\le  c_s \f 1 {t\sqrt { \log t}}$ for $ t>1$. The value $c_s$ is uniformly in the function $f$ with $\|f\|_{L^1(\mu_n)}=1$ and  in the dimension.
The continuous counter-part of the conjecture is for the  Ornstein-Uhlenbeck semi-group $T_t$ with generator $\Delta-x\cdot \nabla$
$$ \sup_{f \ge 0, \|f\|_{L^1(\gamma_n)}=1} \gamma_n
\left(\Big\{\sigma: T_sf(\sigma) \ge t\Big\}\right)  \le c_s \f 1 {t\sqrt { \log t}}, \qquad t \ge 2,$$
where $\gamma_n\sim N(0, I_{n\times n})$ is the standard $n$-dimensional normal distribution. This was proven to be affirmative in Ball, Barthe, Bednorz, Oleszkiewicz and Wolff \cite{BBBOW13}.
The dimension free best constants were given in Eldan and Lee \cite{EL15:focs} and Lehec \cite{Leh16} where the  key ingredients are:
\begin{itemize}
\item [(1)] For any $g\in L^1(\gamma_n)$ and  any $s>0$,  $\nabla^2 (\log T_s g)\ge -c_s^2\, \mathrm {Id}$,
\item [(2)] For any $g\in L^1(\gamma_n)$ non-negative and with  $\nabla^2
(\log g)\ge - \beta \, \mathrm {Id}$ with a $\beta>0$, one has
$\gamma_n (g\ge t)\le \f {C_\beta}{t \sqrt{\log t}}$ for any  $ t>1$.
\end{itemize}
Here  $\mathrm {Id}$ is the identity operator. Such estimates for non-Gaussian measures and also for the $M/M/\infty$ queue on $\N$ were obtained  by
Gozlan, Li, Madiman, Roberto and Samson \cite{GLMRS}.

\section{Main Results}

The short time and asymptotic estimates are presented in \eqref{e1-1}--\eqref{e1-3b} below.  To the best of  our knowledge,
such estimates  were obtained only  for a Riemannian manifold with bounded geometry including a compact Riemannian manifold.
 Gradient and Hessian estimates of the form  (\ref{e1-1}-\ref{e1-1a}) were proved by Sheu \cite{Sh} for $\R^n$ with a non-trivial Riemannian metrics where the objective was a non-degenerate parabolic PDEs  with bounded derivatives up to order three,  and  (\ref{e1-1}) for a compact Riemannian manifold can be found in Driver \cite{D2}, obtained using a result of Hamilton \cite{Hamilton93}, Corollary 1.3 and the Gaussian bounds on heat kernels, see e.g. Li and Yau  \cite{Li-Yau}, Cheeger and Yau \cite{Cheeger-Yau}, Davies \cite{Davies}, Setti \cite{Setti}, and Varopoulos \cite{Varopoulos1,Varopoulos2}.  The  estimate (\ref{e1-1a})  was shown in Hsu \cite{HsuEstimates} again for the compact case.
 For a non-compact Riemannian manifold with non-negative Ricci curvature,
\eqref{e1-1}
was obtained by Kotschwar \cite{Kotschwar}.
Under a bounded geometry condition together with a volume non-collapsing condition, similar estimates were obtained by Souplet and Zhang \cite{SZ} and Engoulatov \cite{En}.
For the heat kernel associated with the Witten Laplacian operator, these estimates  were proved by X.D. Li  \cite{XLi}
under a bounded geometry condition on the Bakry-Emery Ricci curvature.  In addition, in all the references mentioned above, suitable bounded geometry conditions were required. Likewise, the bounded geometry restrictions are used  to derive  differential Harnack inequalities and global heat kernel estimates,  by Cheeger, Gromov and Taylor \cite{Cheeger-Gromov-Taylor}, Cheng, P. Li and Yau \cite{CLY}, Hamilton \cite{Hamilton93},
P. Li and Yau \cite{Li-Yau}, they provide an important step toward \eqref{e1-1}--\eqref{e1-1a}.
Meanwhile, the asymptotic gradient estimate \eqref{e1-2b} was first shown in Bismut \cite{Bis} for a compact
Riemannian manifold. It was  extended to the hypo-elliptic heat kernel and the heat kernel on a vector bundle, for $M$ with bounded geometry, respectively by Ben Arous \cite{Ben}, Ben Arous and L\'eandre \cite{Ben-Le} and Norris \cite{Norris}, c.f. also Azencott \cite{Azencott-as}. 

The asymptotic second order gradient estimate \eqref{e1-3b} was established by Malliavin and Stroock \cite{MS} for
 a compact Riemannian manifold. For `asymptotically flat' Riemannian manifolds with poles  and bounded geometry this can be found in  Aida \cite{A3}.
On cut-locus estimates was studied by  Neel \cite{Neel07}.

A natural question is then whether the estimates \eqref{e1-1}--\eqref{e1-3b} still hold for a general non-compact Riemannian manifold?
Note that in Azencott \cite{Azencott},  it was illustrated  that
 Gaussian type heat kernel  estimates could not be automatically  extended to an arbitrary manifold and may fail if the completeness of the Riemannian metric was removed.

We state the main estimate. For any $y\in M$, let Cut(y) be the cut-locus of $y$ and $i(y)$ be the injectivity radius of $y$.

\begin{theorem}[Theorems \ref{thm6.7} and \ref{thm6.10}] \label{main}
Suppose that $M$ is a complete  Riemannian manifold with  Riemannian distance $d$. 

\begin{itemize}
\item [(1)]
For every compact subset $K$ of $ M$, the following statements hold.

\begin{itemize}
\item [(a)] There exists a
positive constant $C(K)$, which may depend on $K$, such that
\begin{eqnarray}
\label{e1-1}
\quad  \left|\nabla_x \log p(t,x,y)\right|_{T_x M}
&\le& C(K)\left(\frac{1}{\sqrt{t}}+\frac{d(x,y)}{t}\right),\\
\left|\nabla_x^2 \log p(t,x,y)\right|_{T_x M \otimes T_x M}&\le& C(K)\left(\frac{d^2(x,y)}{t^2}+\frac{1}{t}\right)\label{e1-1a}
\end{eqnarray}
for any $ x,y \in K $ and for any $\ t\in (0,1]$.
\item [(b)]  For each $y\in M$ and $\delta<i(y)$ there exist positive constants
$t_0$ and $C_1$ such that
\begin{equation}\label{eq3.29}\aligned
&\left|t\nabla_x^2 \log p(t,x,y)+\textbf{I}_{T_x M}\right|_{T_x M \otimes T_x M}\\
&\le C_1\left(d(x,y)+
\sqrt{t}\right),\quad \quad \ x\in B_y(\delta),\ t \in (0,t_0],\endaligned
\end{equation}
where $\textbf{I}_{T_x M}$ is the identical map on $T_x M$.
\end{itemize}

\item [(2)]
Let $y\in M$ and assume that $\tilde K \subset M\setminus \text{Cut}(y)$ is a compact set. Then
\begin{eqnarray}
\label{e1-2b}
&\lim_{t \downarrow 0}\sup_{x \in \tilde K}
\left|t\nabla_x\log p(t,x,y)+\nabla_x\left(\frac{d^2(x,y)}{2}\right)
\right|_{T_x M}=0,\\
&\lim_{t \downarrow 0}\sup_{x \in \tilde K}
\left|t\nabla_x^2\log p(t,x,y)+\nabla_x^2\left(\frac{d^2(x,y)}{2}\right)
\right|_{T_x M \otimes T_x M}=0.
\label{e1-3b}
\end{eqnarray}
\end{itemize}

\end{theorem}

{\bf Remarks on the main theorem.}
As explained in  Section 1, these estimates are crucial for the stochastic analysis of the loop space $L_o(M)$.
Despite of the collective efforts, so far,  these type of results have been largely proved only for based manifolds with bounded geometry.
While in this paper, we only need to assume that the based manifold $M$  is complete and stochastically complete.
For analysis on the path space $P_o(M)$ over a general complete Riemanian manifold without curvature conditions, some work have already been done by Chen and Wu \cite{CW} and Hsu and Ouyang \cite{Hsu-Ouyang}. For $P_o(M)$, the content of Theorem \ref{main} is not essential.
In a forthcoming paper \cite{CLW}, we shall apply these to obtain integration by parts formula and  construct of O-U Dirichlet form
on $L_o(M)$, and to prove several  functional inequalities on $L_o(M)$.

Our main idea is to obtain  localised asymptotic comparison theorems for the first and the second order gradients of logarithmic heat kernel (see Proposition \ref{prp6.6} and \ref{prp6.9} below).
One novelty  is a new second order derivative formula via  a new type of (second order) stochastic variation for Brownian paths on the orthonormal frame bundles, which is in particular different from that used by Bismut \cite{Bis} or Stroock  \cite{S}.
The idea of stochastic variation was initiated in \cite{Bis} for obtaining an integration by part formula. While the choice of the variation in \cite{S} will produce a term with  (the time reverse of) a non-random vector field on $L^{2,1}(\Omega;\R^n)$, see also Malliavin and Stroock \cite[(1.5)]{MS}, it seems not
possible to replace the non-random vector field in their paper by a random one (otherwise the time reversed field is not adapted, hence It\^o's integral is not well defined), which  prevents the extension of
 the formula in \cite{MS} to  a general non-compact $M$ by a suitable localisation argument.
We shall choose a  variation (see Section \ref{section2} below) with desired properties, which in particular ensures that  the formula for the second order gradient of heat semigroup can take  a random vector fields. This is the key  step for us to extend the new formula to  a general complete $M$ (see e.g. Theorem \ref{thm3.1} below).
The expression we obtain for the second order gradient of heat semigroup is different  from that by Elworthy and Li   \cite{EL}, Li \cite{Li-doubly-damped, Li18}, or from that in
Arnaudon, Plank and Thalmaier \cite{APT} or that in Thompson \cite{Thompson}. We prove the formula by combining the second order stochastic
variation (shown to hold for a compact manifold) and approximation arguments (for a non-compact manifold), which is totally different from that in \cite{APT,Thompson}. This new method is adapted for both the proof of Proposition \ref{prp6.9} here and the integration by parts formula in
our forthcoming paper~\cite{CLW}.
\medskip

\section{Expression for the second order gradient of heat semigroup}

Throughout the paper, $(\Omega, \mathscr{F}, \mathscr{F}_t, \mathbb{P})$ denotes a filtered probability space satisfying the standard assumptions, and $B_t=(B_t^1,B_t^2,\cdots, B_t^n)$ is a standard $\R^n$-valued Brownian motion.
Let $L(\Omega;\R^n)$ denote the collection of all stochastic processes $h: \R_+\times\Omega \rightarrow \R^n$ which are $\mathscr{F}_t$-adapted.
Let
$h'(\cdot, \omega)$ denote the time derivative of $h(\cdot, \omega)$. We define the Cameron-Martin space on the Wiener space as follows
 \begin{equation*}\aligned
L^{2,1}(\Omega;\R^n):&=\bigg\{
h\in L(\Omega;\R^n):
 \; h(\cdot,\omega)~\text{is absolutely continuous for a.s.}~\omega\in \Omega,\\& ~~~~~~~~~~~~~~~~~~~~~~~~~~~~~~~~~~~~~~\text{and}~\mathbb{E}\Big[\int_0^1 |h'(s, \omega)|^2\, \d s\Big] <\infty\bigg\}.
\endaligned\end{equation*}
Elements of $L^{2,1}(\Omega;\R^n)$ are usually called (random) Cameron-Martin vectors.
Let
$C_b(M)$ and $C_c(M)$ denote the collection of all real valued bounded and continuous functions on $M$ and continuous functions
with compact supports in $M$ respectively. Let $\so(n)$ denote the set of of anti-symmetric $n\times n$ matrices and let $SO(n)$  denote the collection
of  orthonormal $n\times n$ matrices.

{\bf The curvature.} Let ${\bf R}_x$ denote the sectional curvature tensor and let $\Ric_x$ denote the Ricci curvature tensor at $x\in M$ respectively. Thus both ${\bf R}_x:T_x M\times T_x M \rightarrow T_x M \times T_x M$ and $\Ric^{\sharp}_x:T_x M \rightarrow T_xM$ are linear map,  the latter is given by the duality:
$$\big\langle \Ric^{\sharp}_x(v_1), v_2\big\rangle_{T_x M} =\Ric_x(v_1,v_2),\quad\ \forall\ v_1,v_2\in T_x M.$$

\medskip

{\bf The horizontal Brownian motion.}
Given a point $x\in M$, let $O_xM$ denote the space of  linear isometries from $\R^n$ to $T_xM$. Let  $OM:=\cup_{x\in M} O_xM$, which is the orthonormal frame bundle over $M$,  and  let $\pi: OM \rightarrow M$ denote the canonical projection which takes a frame $u\in O_x M$ to its base point $x$.
For every $u\in OM$, we define
${\rm R}_u:\R^n\times \R^n \rightarrow \R^n\times \R^n$ and $ {\rm ric}_u:\R^n \rightarrow \R^n$ by
$$\begin{aligned}{\rm R}_u(e_1,e_2):&=u^{-1}\big({\bf R}_{\pi(u)}(ue_1,ue_2)\big), \\
\ric_u(e_1):&=u^{-1}\big({\bf \Ric^{\sharp}}_{\pi(u)}(ue_1)\big)
\end{aligned}$$
for every $ e_1,e_2\in \R^n$.

Given a vector $e\in \R^n$, we denote by $H_{e}$ the associated canonical horizontal vector field on $OM$ with the property that $(T\pi)_{u}(H_{e})=u e\in T_{\pi(u)}M$.
Thus the solution of the ODE $$u'(t)=H_e(u(t))$$ projects to the geodesic on $M$ with the initial position $x$ and the initial speed $u(0)(e)$.

We choose an orthonormal basis $\{e_i\}_{i=1}^n$ of $\R^n$.  
Suppose
$\{U_t\}_{t\ge 0}$ is the solution of following $OM$-valued Stratonovich stochastic differential equation
\begin{equation}\label{sde1}
\d U_t=\displaystyle\sum^n_{i=1}H_{e_i}\left(U_t \right)\circ\d B_t^i,
\end{equation}
where the initial value $U_0$ is a fixed orthonormal basis of $T_xM$.
We usually call $\{U_t\}_{0\le t<\zeta}$ the canonical horizontal Brownian motion, where
$\zeta:\Omega \to \R_+$ is the life time for $U_t$. Let 
$X_t^x:=\pi(U_t)$, $0\le t<\zeta(x)$, then $X_t^x$  is
a Brownian motion
on $M$ with initial value $x$ and life time $\zeta(x)$.
This is the celebrated intrinsic construction of $M$-valued Brownian motion by  Eells and Elworthy \cite{EE} and
Elworthy \cite{Elworthy}, see also Malliavin  \cite{Malliavin}.
It is well known that the Brownian motion on $M$  does not explode if and only if the horizontal Brownian motion $U_t$ on $OM$ does not explode.
 In particular, it does not rely on the choice of an isometrically embedding from $M$ to an ambient Euclidean space.
Let $$P_tf(x):=\Ee\left[f\left(X_t^x\right)\1_{\{t<\zeta(x)\}}\right]$$
 be the heat semigroup associated to Brownian motion $X_{\cdot}$.

\emph{The superscript $x$ may be omitted if there is no risk of confusion.}

\subsection{Second order gradient of the heat semigroup}
Let $\{U_t\}_{0\le t <\zeta(x)}$
denote  the horizontal Brownian motion on $M$ and
$\{X_t^x=\pi(U_t)\}_{0\le t<\zeta(x)}$ is the Brownian motion on $M$ with  initial value $x$ and  life time $\zeta(x)$.
For any $h\in L^{2,1}(\Omega;\R^n)$, we set
\begin{equation}\label{e3-2}
 \Gamma^h_t:=\int_0^t {\rm R}_{U_s} \left( \circ \d B_s, h(s)\right),\qquad
 \Theta^h_t:=h'(t)+\frac{1}{2}\ric_{U_t}(h(t)),
\end{equation}
It is easy to see that $\Gamma^h_t$ is an $\so(n)$-valued process.  For $ t\geq0$,  we define
\begin{equation}\label{e3-2a}
\begin{split}
\Lambda_t^h:=\Gamma_t^h h'(t)+\frac{1}{2}U_t^{-1} \; \nabla \Ric^{\sharp}_{X_t}\big(U_th(t), U_t h(t)\big)
-\f 12\Gamma_t^h \;\ric_{U_t}(h(t))+\f 12 \ric_{U_t}\left(\Gamma_t^h h(t)\right).
\end{split}
\end{equation}

We are now ready to state one of our main tools,  the second order gradient formula on a general complete $M$.
\begin{thm}\label{thm3.1}
 Suppose that $M$ is a complete Riemannian manifold.  Let $\{D_m\}_{m=1}^\infty$ denote the increasing family of exhaustive relatively compact open sets of $M$ and
let $\{l_m\}_{m=1}^\infty$  denote the cut-off vector fields
as constructed in Lemma \ref{lem5.1}.
Let $x\in m$, and there exists $m_0\in \N$  such that $x\in D_{m_0+1}$.

For every $m> m_0$, $v \in T_x M$, and $t\in (0,1]$, we define
$$h(s):=\Big(\frac{t-2s}{t}\Big)^+\cdot l_m\left(s,X_{\cdot}^x\right)\cdot U_0^{-1}v,
\quad s\geq0.$$
Then  $h\in L^{2,1}(\Omega;\R^n)$. Furthermore, for any $f\in C_b(M)$ we have
\begin{equation}\label{t3-1-1}
\begin{split}
&\big\langle \nabla^2 P_t f(x), v\otimes v\big\rangle_{T_x M \otimes T_x M}\\
&=\EE_x\left[\left(
\left(\int_0^t\langle \Theta^h_s, \d B_s\rangle\right)^2-
\int_0^t \langle \Lambda^h_s, \d B_s\rangle-\int_0^t \left|\Theta^h_s\right|^2 \d s\right)f(X_t^x)\1_{\{t<\zeta(x)\}}\right].
\end{split}
\end{equation}
\end{thm}

In particular, the processes $l_m(t, \gamma)$  equals to $1$ at any time before $\gamma$ exits $D_{m-1}$ and equals to zero after it exits $D_m$ for the first time. So it is obvious to see that $h(t, \gamma)=U_0^{-1}v$ at $t=0$ and vanishes after the first exit time of $\gamma$ from $D_m$.

\subsection{Comments}

The main idea for proving the second order gradient of the heat semigroup $P_t$ is to
approximate the formula on $M$ by those for a family of specific compact manifolds.
We first use a result of Greene and Wu \cite{GW} to construct a family of relatively compact exhausting open subsets $\{D_m\}_{m=1}^\infty$,
which is valid for a complete Riemannian manifold $M$. This allows to construct a series of random cut-off vector fields $l_m\in L^{2,1}(\Omega;\R^n)$
vanishing,  as soon as  the sample path exits $D_m$ for the first time, with the necessary quantitative estimates needed for the localisation. See Lemma \ref{lem5.1} below for details.
The lemma is partly inspired by the work of Thalmaier \cite{T} and  Thalmaier and Wang \cite{TW}, where geodesic balls are used. For the purpose of embedding into compact manifolds,  we make sure that each $D_m$ having a smooth boundary which, because of the cut locus, cannot be taken as granted of geodesic balls on arbitrary Riemannian manifolds.

We want to remark that this  offers a more powerful (and also a more reliable) alternative to localisation with stopping times, the latter has been commonly used in stochastic calculus and occasionally incorrectly used. The stopping time argument relies on a continuity assumption on the Brownian motion with respect to the initial value.
   Such continuity  condition seems not easy to verify (for stopping times), and ought not be applied casually, see e.g.  Elworthy \cite{Elworthy78}, Li and Sheutzow \cite{Li-Scheutzow},  and Li \cite{Li-flow} for more details.
   Note, however,  that exit times from regular domains do have good regularity properties in the sense of Malliavin calculus, we refer the reader to
the work of Airault, Maillian, and Ren \cite{Airault-Malliavin-Ren} for more details.

Cut-off vector fields have been  previously applied by
Arnaudon, Plank and Thalmaier \cite{APT}, Thompson \cite{Thompson}, Thalmaier \cite{T}, and  Thalmaier and Wang \cite{TW}
to provide a {\it localised} differential formula
for heat semigroups. As explained earlier, we use a new type of (second order) stochastic variation argument to construct
the global second order gradient formula given below. In particular, the expression here is different from that
of Elworthy and Li \cite{EL}, Arnaudon, Plank, and Thalmaier \cite{APT}, Li \cite{Li-doubly-damped, Li18} and Thompson \cite{Thompson}
and particularly we do not use the doubly parallel translation operators used in \cite{Li-doubly-damped, Li18}.

\subsection{Comparison theorems}
The outline of the proof is as follows. We first show that the formula holds for a compact Riemannian manifold, this proof is given  in Section \ref{section2} using a new stochastic variation.
To pass from a compact manifold to a non-compact manifold, we use a suitable isometric embedding from $D_m$ into a compact Riemannian manifold $\tilde M_m$, as well as the quantitative cut-off process $l_m$ constructed by Lemma \ref{lem7.1} and Lemma \ref{lem5.1} respectively.

Denote by $p_{\tilde M_m}(t,x,y)$  the heat kernel on $\tilde M_m$.
 Although the heat kernel of a Riemannian manifold is determined in a global manner by the Riemannian metric,   
  we  obtain, below, short time comparison theorems between $\nabla \log p_{\tilde M_m}$, $\nabla^2 \log p_{\tilde M_m}$
and $\nabla \log p$, $\nabla^2 \log p$. These are used for proving \eqref{e1-1}--\eqref{e1-3b}.

The comparison theorem below allows us to obtain estimates for $\nabla \log p$ and $\nabla^2 \log p$, with the successive applications of first order and second order gradient formula as well as comparison estimates for functionals of the Brownian motions on $M$ and that on $\tilde M_m$.

\begin{prp}
(Propositions  \ref{prp6.6} and  \ref{prp6.9})\label{prp3.1}
 Suppose $K$ is a compact subset of $M$. For any constant $L>1$, there exists a $m_0=m_0(K,L)\in \N$, which may depend on $K$ and $L$,  such that for all $m\ge m_0$ we could find a positive time
$t_0=t_0(K,L,m)$ such that
\begin{eqnarray*}
&\sup_{x,y\in K} \left|\nabla_x\log p(t,x,y)- \nabla_x \log p_{\tilde M_{m}}(t,x,y)\right|_{T_x M}
\le C(m)e^{-\frac{L}{t}},\quad \forall\ t\in (0,t_0],\\
&\sup_{x,y\in K} e^{\frac{L}{t}}  \left|\nabla_x^2\log p(t,x,y)-\nabla_x^2 \log p_{\tilde M_{m}}(t,x,y)
\right|_{T_x M\otimes T_x M}\le C(m)e^{-\frac{L}{t}},\quad \forall\ t\in (0,t_0],
\end{eqnarray*}
where $C(m)$ is a positive constant  depending on $m$.
\end{prp}

\section{Second Order Variation on a Compact Manifold}\label{section2}

\quad \quad \emph{Throughout this section,  $M$ is an $n$-dimensional compact Riemannian manifold.}
In Proposition \ref{prp4.5} below, we shall establish (\ref{t3-1-1}) for a compact manifold, which is a  fundamental step  toward Theorem \ref{thm3.1}.

 The first second order differential formula for the heat semigroup $P_t$ was obtained by Elworthy and Li \cite{EL1}  for a non-compact manifold, however with restrictions on their curvature. Another disadvantage of the formula was its involvement of a non-intrinsic curvature which was due to the application of the derivative flow of gradient  stochastic differential equations, as well as a martingale approach developed in  Li \cite{Li-thesis}.   An intrinsic formula for $\nabla^2 P_t f$ was given by Stroock  \cite{S} for a compact Riemannian manifold,  while a localised intrinsic formula  was obtained by Arnaudon, Plank and Thalmaier \cite{APT} with the martingale approach.
The study of the second order gradient of  the Feynman-Kac semigroup of an operator $\Delta+V,$ with a potential function, was pioneered by Li \cite{Li18, Li-doubly-damped}, where  a path integration formula was obtained with the help of doubly damped stochastic parallel transport equation. (The  first order gradient formula was previously obtained in  Li and Thompson \cite{Li-Thompson}, c.f. \cite{EL1, EL-Vilnius}.)
A  localised version of the Hessian formula (still with doubly stochastic damped parallel translations) for the  Feynman-Kac semigroup  was derived by Thompson \cite{Thompson}.

However, all the expressions  mentioned earlier do not seem to lead to our application, such as the proof of Proposition \ref{prp3.1}. To overcome this problem  we introduce a quantitative localisation procedure and obtain a second order gradient formula to which this localisation method can be applied.

One of our main tools is to extend Bismut's idea to perturb the $M$-valued Brownian motion with initial value $\xi(\varepsilon)$ (where
$\xi(\varepsilon)$ is a smooth curve in $M$), they will be constructed as solutions of a family of SDEs with the driving Brownian motion
$\{B_t\}_{t\ge 0}$ rotated and translated appropriately.
The rotation and translation exerted on 
$\{B_t\}_{t\ge 0}$
transmits the variation in the initial value of the Brownian motion on the manifold  to variations, in the same parameter, of the Radon Nikodym derivatives of a family of probability measures, with respect to which the solutions are Brownian motions on $M$. This simple and elegant idea was applied in Bismut \cite{Bis} for deducing an integration by parts formula. Incidentally, such integration by parts formula and the first order gradient formula of the heat semigroup were proved to be equivalent on a compact manifold by Elworthy and Li \cite{EL}.
 In  Stroock \cite{S}, by calculating the concrete form of the second variation introduced by Bismut,  this idea was adapted for obtaining the
second order derivative formula for the heat semigroup on a compact  manifold.
As explained earlier,
the choice of stochastic variation in  \cite{S} (see also Malliavin and Stroock \cite[(1.5)]{MS}) will produce a term coming from  the time
reverse of a non-random vector field on $L^{2,1}(\Omega;\R^n)$,  and it seems not possible to replace the non-random vector field  by a random one (otherwise the time reversed field is not adapted,  hence It\^o's integral is not well defined). Therefore the formula obtained
in Stroock \cite{S} may not be extended to the one with a random vector field and so is not suitable to for extension to non-compact manifolds with the localisation technique we  introduce shortly.

One crucial ingredient for our choice of the stochastic variation is that
it ensures \eqref{l3-3-1}, which implies that the second variation vanishes at time $t$ when we choose
a vector field $h$ in the translated part satisfying $h(t)\equiv 0$. This allows us to derive a second order gradient formula
with localised vector fields and to extend it to a general (non-compact) complete Riemannian manifold.

\subsection{A novel stochastic variation with a second order term}

As before,  $\{U_t\}_{0\le t<\zeta(x)}$ is the solution of equation \eqref{sde1} with initial point $U_0$ and $\pi(U_0)=x$.
In Bismut \cite{Bis} the following classical perturbation for the driving force $B_t$ was used:
 $$\hat B_t^\varepsilon=\int_0^t e^{-\varepsilon \Gamma^h_s}\d B_s+\varepsilon  \int_0^t \left(h'(s) +\f 12 \ric_{U_s} h(s)\right) \d s.$$
where $h\in L^{2,1}(\Omega;\R^n)$ is a chosen Cameron-Martin vector and  $\Gamma^h_t:=\int_0^t {\rm R}_{U_s} \left( \circ \d B_s, h(s)\right)$.
This perturbation of the noise
works well with the first variation for which one needs to ensure that $\f {\partial} {\partial\varepsilon}|_{\varepsilon =0} \pi (U_t^\varepsilon) =
U_th(t)$ and has been the popular and standard perturbation, as used also in Driver \cite{D1}, Fang and Malliavin \cite{FM}. Other variation of the noise
are also of first order perturbations.

However, with the above mentioned variation,  $\f {\partial^2} {\partial\varepsilon^2}|_{\varepsilon =0} \pi (U_t^\varepsilon) \not =0$ as long as $h(t)\not \equiv 0$.
To solve this problem, we will introduce a second order variation (such perturbation  is not unique and  we may find a slightly different choice).  Unlike the case with the classical perturbation, this time we cannot avoid differentiating the structure equation so  have to choose a connection on the frame bundle. Our approach is inspired by the theory of linear connections induced by a SDE developed by Elworthy, LeJan and Li \cite{ELL}.  We believe that the same method can also be used for higher order
variations.

For any $h\in L^{2,1}(\Omega;\R^n)$, we have defined an $\so(n)$-valued process $\Gamma_t^h$ and $\R^n$-valued process
$\Theta_t^h$, $\Lambda_t^h$ by \eqref{e3-2} and \eqref{e3-2a} respectively.
We first introduce the translation and define the $\R^n$-valued
process $B_t^{\e,h}$ as follows
\begin{equation}
B_t^{\varepsilon,h}:=B_t+\e\int_0^t h'(s)\d s+\frac{\e^2}{2}\int_0^t \Phi^h_s\d s,
\end{equation}
where  $\Phi^h_t:=\Gamma^h_t h'(t)$.
We then introduce a rotation for $\R^n$-valued Brownian motion. Let us first
 set
\begin{equation}\label{e3-3}
\begin{aligned}
\Gamma_t^{(2),h}:=&\int_0^t U_s^{-1}\nabla {\rm {\bf R}}_{\pi(U_s)}\big(U_s h(s), U_s\circ \d B_s, U_s h(s)\big)-\int_0^t
\Gamma^h_s{\rm R}_{U_s}(\circ \d B_s, h(s))\\
&+\int_0^t {\rm R}_{U_s}(h'(s),h(s))\d s+\int_0^t {\rm R}_{U_s}\left(\circ \d B_s, \Gamma^h_s h(s)\right).
\end{aligned}
\end{equation}
It is easy to see that $\Gamma_t^{(2),h}$ is an $\so(n)$-valued process.
Then for every $\varepsilon>0$, we define $SO(n)$-valued process $G_t^{\e,h}$ as follows
\begin{equation*}
\begin{aligned}
G_t^{\varepsilon,h}&:=\exp\left({-\varepsilon \Gamma^h_t
-\frac{\varepsilon^2}{2} \Gamma_t^{(2),h} }\right),
\end{aligned}
\end{equation*}
where $\exp:\so(n)\to SO(n)$ is
the exponential map in the Lie algebra $\so(n)$ of $SO(n)$.

 We can now introduce $\tilde B_t^{\varepsilon,h}$, the variation of $B_t$, as well as the corresponding equation on~$OM$.
\begin{defn}\label{def4.1}
Let $\xi(\e)$, $\e\in(-1,1)$, be a geodesic with $\xi(0)=x$.
Let $\{U_0^{\varepsilon,h}: \e\in(-1,1)\}$ be a parallel orthonormal
 frame along $\xi(\e)$ with $\pi(U_0^{\varepsilon,h})=\xi(\e)$. Let $U_t^\epsilon$ denote the
 solution of the following equation with initial condition $U_0^{\varepsilon,h}$,
\begin{equation}\label{sde2}
\begin{aligned}
&\d U_t^{\varepsilon,h}=\sum_{i=1}^n H_{e_i}(U_t^{\varepsilon,h})\circ \d \tilde B_t^{\e,h,i},\\&
\d\tilde B_t^{\varepsilon,h} =G_t^{\varepsilon,h}\circ \d B_t^{\e,h},\quad \tilde B_0^{\varepsilon,h}=0.
\end{aligned}
\end{equation}
We define $X_t^{\varepsilon,\xi(\e),h}= \pi(U_t^{\varepsilon,h})$. If $\e=0$, then $X_t^{0,x,h}=X_t^x$ with
 $X_t^x=\pi(U_t)$.
 \end{defn}
We remark that the perturbation in $U_t^{\varepsilon,h}$  has a translation part $B_t^{\varepsilon,h}$, and a rotation part $G_t^{\varepsilon,h}$. The rotation
$G_t^{\varepsilon,h}$ is chosen to offset precisely the twisting effects induced by the second
order stochastic variation.

 \emph{For simplicity we omit the subscript
$h$, in $\Theta^h_t,\Lambda^h_t$,
$X_t^{\varepsilon,h},\Gamma_t^{h},\Gamma_t^{(2),h},G_t^{\varepsilon,h}$, $B_t^{\varepsilon,h}$  and $U_t^{\varepsilon,h}$, from time to time.}

Let $\varpi$ and $\theta$ denote respectively the $\so(n)$-valued connection $1$-form and the $\R^n$-valued solder $1$-form
respectively.
 Set
 $$\varpi_t^\varepsilon:=
 \varpi\left( \f \partial {\partial \varepsilon}U_t^\varepsilon\right),  \qquad
 \theta_t^\varepsilon:= 
 \theta\left( \f \partial {\partial \varepsilon} U_t^\varepsilon\right).$$

Through this paper, we use
$D_t$, $\d_t$ to denote the stochastic covariant differential  for vector fields
and stochastic differential on $M$ along a semi-martingale respectively and
$\frac{D}{\partial \e}$ denotes the covariant derivative for vector fields on $M$ with respect to the variable $\e$.

\begin{lem}\label{lem 4.1}
If we choose $h\in L^{2,1}(\Omega;\R^n)$ such that $h(0)=U_0^{-1}\left(\frac{\partial}{\partial \e}\Big|_{\e=0}\xi(\e)\right)$, then
 \begin{equation}\label{varpi}
\varpi_t^\varepsilon=\int_0^t  {\rm R}_{U_s^\e} (G_s^\varepsilon\circ \d B_s^\varepsilon, \theta_s^\varepsilon).
\end{equation}
And $\theta_t^\e$ satisfy the following equation,
\begin{equation}\label{l3-1-1}
\begin{cases}
& \d\theta_t^\e=-\big(\Gamma_t+\e\Gamma^{(2)}_t\big)G_t^\e\circ \d B_t^\e+\w_t^\e G_t^\e\circ \d B_t^\e+
G_t^\e\big(h'(t)+\e\Phi_t\big)\d t,\\
& \theta_0^\e=(U_0^\e)^{-1}\frac{d \xi(\e)}{d \e}.
\end{cases}
\end{equation}
In particular, we have
\begin{equation}
\left\{\begin{aligned}
&\theta_t^0:=\theta\left( \f \partial {\partial \varepsilon} \Big|_{\varepsilon=0}U_t^\varepsilon\right)=h(t), \label{l3-1-2}
\\
&\w^0_t=\Gamma_t,\\
&\frac{D}{\partial \e}\Big|_{\e=0}\left(U_t^\e G_t^\e e\right)=0,\quad \ \forall\ e\in \R^d,\\
&\frac{\partial X_t^\e}{\partial \e}\Big|_{\e=0}=U_t h(t).
\end{aligned}\right.
\end{equation}
 \end{lem}

 \begin{proof}
We first use the structure equation
\begin{align*}
\d\varpi\Big(\d_t U_t^\varepsilon, \f \partial {\partial \varepsilon}U_t^\varepsilon\Big)&=-\varpi\wedge \varpi
\left(\d_t U_t^\e, \f \partial {\partial \varepsilon}U_t^\varepsilon \right)+{\rm R}_{U_t^\e}\left(\theta\left(\d_t U_t^\varepsilon\right),
\theta\left( \f
\partial {\partial \varepsilon}U_t^\varepsilon\right)\right)\\
&=-\sum_{i=1}^n\varpi\wedge \varpi\left(H_{e_i}(U_t^{\e})\circ d\tilde B^{\e,i}_t, \f \partial {\partial \varepsilon}U_t^\varepsilon\right)+{\rm R}_{U_t^\e}\left(\theta\left(\d_t U_t^\varepsilon\right),\theta_t^\varepsilon\right)\\
&={\rm R}_{U_t^\e}\left(\theta\left(\d_t U_t^\varepsilon\right),\theta_t^\varepsilon\right)
\end{align*}
to obtain
\begin{align*}
 \d \varpi_t^\e
&=\d\varpi\left(\d_t U_t^\varepsilon, \f \partial {\partial \varepsilon}U_t^\varepsilon\right)={\rm R}_{U_t^\e}\left(\theta\left(\d_t U_t^\varepsilon\right),
\theta_t^\varepsilon\right)={\rm R}_{U_t^\e}\left(G_t^\e\circ \d B_t^\e, \theta_t^\e\right).
 \end{align*}
Since at time $0$, the variation $\{U_0^\e;\e\in (-1,1)\}$ is parallel along the geodesic $\xi$,  $\varpi_0^\e=0$.  Then
\eqref{varpi} follows immediately.

Here we have used the Transfer Principle: on the compact manifold $M$ we could \emph{treat the Stratonovich integral as the ordinary derivative} (with respect to time variable)
in the computation.  Crucially we could \emph{exchange the order of differentiations and integrations. } The transfer principle is well known
for \emph{compact} manifolds, see e.g.
\cite{FM} or
\cite{MalliavinMon}, but not automatically apply to non-compact manifolds nor automatically to the less smooth case nor to the derivative processes. This is used in similar computations later in the article without further comment.

Due to the torsion free property,
the time derivative and the derivative for $\varepsilon$ could commute: $D_t\frac{\partial }{\partial \e}=
\frac{D}{\partial \e}\d_t$.
Also note that $\theta_t^\e=(U_t^{\e})^{-1}T\pi ( \f \partial {\partial \e} U_t^\e)$, so we have,
\begin{equation}\label{1derivative}
\begin{split}
{\d } \theta_t^\e&=(U_t^\e)^{-1}\left(D_t \left(\frac{\partial}{\partial \e}X_t^\e\right)
\right)=(U_t^\e)^{-1}\left(\f D {\partial \e}\d_t  X_t^\e\right)\\
&=(U_t^\e)^{-1}\left(\frac{D}{\partial \e}\left(U_t^\e G_t^\e \circ \d B_t^\e\right)\right)\\
&=\w_t^\e G_t^\e  \circ \d B_t^\e +\frac{\partial G_t^\e}{\partial \e}\circ \d B_t^\e+G_t^\e \circ
\d\left(\frac{\partial}{\partial \e}B_t^\e\right)\\
&=\w_t^\e G_t^\e  \circ \d B_t^\e-\big(\Gamma_t+\e\Gamma^{(2)}_t\big)G_t^\e\circ \d B_t^\e+
G_t^\e\big(h'(t)+\e\Phi_t\big)\d t,
\end{split}
\end{equation}
 where the fourth equality is due to
  \begin{equation}
 \label{proof-1-4}
\begin{aligned}
\f D {\partial \e} \left( U_t^\e G_t^\e \right)
&=U_t^\e \left(  \varpi_t^\e G_t^\e + \f {\partial}{\partial\e} G_t^\e \right).\end{aligned}
\end{equation}
So we have obtained the first equation
in \eqref{l3-1-1}. The initial condition in  \eqref{l3-1-1} follows trivially from the fact
$ \theta_0^\varepsilon=(U_0^{\e})^{-1}\pi ( \f \partial {\partial \e} U_0^\e))$, $\{U_0^\e; \e\in (-1,1)\}$ is a parallel orthonormal frame bundle along $\xi(\cdot)$ and $X_0^\e=\xi(\e)$.

Based on the fact that
$$\varpi_t^0=\int_0^t {\rm R}_{U_s} \left( \circ \d B_s, \theta_s^0\right), \qquad \Gamma_t=\int_0^t {\rm R}_{U_s} \left( \circ \d B_s, h(s)\right),$$
and taking $\e=0$ in \eqref{l3-1-1} we arrive at
\begin{align*}
d\theta^0_t=\left( \int_0^t {\rm R}_{U_s} \left( \circ \d B_s, \theta_s^0\right)-\int_0^t {\rm R}_{U_s}\left(\circ \d B_s,h(s)\right)
\right)\circ \d B_t+ h'(t)\d t, \quad \ \theta_0^0=h(0).
\end{align*}
It is easy to verify that $\theta^0_t=h(t)$ is the unique solution to above equation, proving the first line of \eqref{l3-1-2}.
Then plugging in $\theta_t^0=h(t)$ into   \eqref{varpi} to see that $\w^0_t=\Gamma_t$, so we have
\begin{align*}
\f D {\partial \e}\Big|_{\e=0}\left(U_t^\e G_t^\e e\right)=
U_t\left(\w_t^0 e+\frac{\partial}{\partial \e}\Big|_{\e=0}G_t^\e e\right)=U_t\left(\Gamma_t e-\Gamma_t e\right)=0,
\end{align*}
which is  the third line of \eqref{l3-1-2}. Finally, $D_t\Big(\frac{\partial X_t^\e}{\partial \e}\Big|_{\e=0}\Big)=U_t \d \theta_t^0=U_t h'(t)\d t$, giving
$ \frac{\partial X_t^\e}{\partial \e}\Big|_{\e=0}=U_th(t)$. This completes the proof.
\end{proof}
In particular, we obtain the following lemma:

\begin{lem}\label{lem 4.2}
For every $h\in L^{2,1}(\Omega;\R^n)$ with $h(0)\equiv v=U_0^{-1}\left(\frac{\partial}{\partial \e}\Big|_{\e=0}\xi(\e)\right)$,
we have
\begin{equation}\label{l3-2-1}
\frac{\partial }{\partial \e}\Big|_{\e=0}\w_t^\e=\int_0^t {\rm R}_{U_s}\left(\circ \d B_s,\eta_s\right)+\Gamma_t^{(2)},
\end{equation}
where $\eta_s:=\f {\partial \theta_s^\varepsilon} {\partial \varepsilon} \Big|_{\varepsilon=0}$ and $\Gamma_t^{(2)}$ is defined by \eqref{e3-3}.
\end{lem}

\begin{proof}
By the first line of \eqref{l3-1-2} we have $\theta_t^0=h(t)$. We differentiate the integral expression \eqref{varpi}
for $\w_t^\e$ and apply the third line of \eqref{l3-1-2} to obtain
\begin{align*}
  \frac{\partial }{\partial \e}\Big|_{\e=0}\w_t^\e
  =&\frac{\partial}{\partial\varepsilon}\Big|_{\varepsilon=0}
  \int_0^t  (U_s^{\e})^{-1}
  {\rm { R}}_{X_s^\e}\left( U_s^\e G_s^\varepsilon\circ \d B_s^\varepsilon, U_s^\e \theta_s^\varepsilon\right)
  \\
  =&
  \int_0^t \frac{\partial}{\partial\varepsilon}\Big|_{\varepsilon=0}\left(G_s^\e  ( U_s^{\varepsilon}G_s^\varepsilon )^{-1}
 {\rm { R}}_{X_s^\e} \left(U_s^\varepsilon G_s^\varepsilon\circ \d B_s^\varepsilon, U_s^\varepsilon G_s^{\e}(G_s^{\e})^{-1}\theta_s^\varepsilon\right)
  \right)\\
  =&\int_0^t \left(\frac{\partial G_s^\e}{\partial \e}\Big|_{\e=0}\right) {\rm R}_{U_s}  \left(\circ \d B_s,\,\theta_s^0\right)+
  \int_0^t U_s ^{-1}\nabla {\rm { R}}_{X_s}\left(U_s\theta_s^0, U_s\circ \d B_s, U_s \,\theta_s^0\right)
  \\
  &+\int_0^t {\rm R}_{U_s}\left(\circ \d \frac{\partial B_s^\e}{\partial \e}\Big|_{\e=0}, \theta_s^0\right)+
  \int_0^t {\rm R}_{U_s}\left(\circ \d B_s, \left(\frac{\partial (G_s^\e)^{-1}}{\partial \e}\Big|_{\e=0}\right)\theta_s^0\right)\\
  &+\int_0^t \, {\rm R}_{U_s}\left(\circ \d B_s, \f {\partial \theta_s^\varepsilon} {\partial \varepsilon} \Big|_{\varepsilon=0}\right).
    \end{align*}
    Here the last term is $\int_0^t {\rm R}_{U_s}\left(\circ \d B_s,\eta_s\right)$, while the sum of the rest  is $\Gamma_t^{(2)}$, so we have
completed the proof.
\end{proof}

We observe that  $\eta_s=\f {\partial \theta_s^\varepsilon} {\partial \varepsilon} \Big|_{\varepsilon=0}$ is essentially the second variation of $\pi(U_s^\epsilon)$.

\begin{lem}\label{lem 4.3}
For every $h\in L^{2,1}(\Omega;\R^n)$ with $h(0)\equiv v=U_0^{-1}\left(\frac{\partial}{\partial \e}\Big|_{\e=0}\xi(\e)\right)$,
we have $\eta_t\equiv 0$ for all $t\in [0,1]$ and
\begin{equation}\label{l3-3-1}
\frac{D}{\partial \e}\Big|_{\e=0}\left(\frac{\partial X_t^\e}{\partial \e}\right)=U_t\Gamma_th(t).
\end{equation}
\end{lem}

\begin{proof}
We recall the first equation of \eqref{l3-1-1}
$$d\theta_t^\e=-\big(\Gamma_t+\e\Gamma^{(2)}_t\big)G_t^\e\circ \d B_t^\e+\w_t^\e G_t^\e\circ \d B_t^\e+
G_t^\e\big(h'(t)+\e \Gamma_th'(t))\d t.$$
Differentiating it at $\e=0$, using   \eqref{l3-2-1} and
the following fact
\begin{equation*}\label{summarise}
\w_t^0=\Gamma_t, \quad \frac{\partial B_t^\e}{\partial \e}\Big|_{\e=0}=h'(t), \quad
\frac{\partial G_t^\e}{\partial \e}\Big|_{\e=0}=-\Gamma_t, \quad  \Phi_t=\Gamma_t h'(t),
\end{equation*}
we could obtain
\begin{equation*}\label{l3-3-2}
\begin{split}
\d\eta_t=&-\left(\Gamma_t^{(2)}+\Gamma_t\frac{\partial G_t^\e}{\partial \e}\Big|_{\e=0}\right)\circ \d B_t
-\Gamma_t \circ \d\left(\frac{\partial B_t^\e}{\partial \e}\Big|_{\e=0}\right)+\left(\w_t^0\frac{\partial G_t^\e}{\partial \e}\Big|_{\e=0}
+\frac{\partial \w_t^\e}{\partial \e}\Big|_{\e=0}\right)\circ \d B_t\\
&+\w_t^0 \d\left(\frac{\partial B_t^\e}{\partial \e}\Big|_{\e=0}\right)+\frac{\partial G_t^\e}{\partial \e}\Big|_{\e=0}h'(t)\d t
+ \Gamma_th'(t)\d t\\
=&\left(\int_0^t {\rm R}_{U_s}\left(\circ \d B_s, \eta_s\right)\right)\circ \d B_t.
\end{split}
\end{equation*}

At the same time, since $X_0^\e=\xi(\e)$, $\xi(\cdot)$ is a geodesic, and also $\{U_0^\e, \e\in (-1,1)\}$ is a parallel orthonormal frame bundle along $\xi(\cdot)$, we could verify that
$$\eta_0=\frac{\partial \theta_0^\e}{\partial \e}\Big|_{\e=0}=U_0^{-1}\left(\frac{D}{\partial \e}\Big|_{\e=0}
\left(\frac{\partial \xi(\e)}{\partial \e}\right)\right)=0.$$
Observe that  the unique solution to following equation is $v_t\equiv 0$ \begin{align*}
\d v_t=\left(\int_0^t {\rm R}_{U_s}\left(\circ \d B_s, v_s \right)\right)\circ \d B_t,\quad \ v_0=0.
\end{align*}
Then we derive that
$\eta_t\equiv 0$ for all $t\in [0,1]$.

Moreover, note that by definition we have $\frac{\partial X_t^\e}{\partial \e}=U_t^\e \theta_t^\e$, due to
the fact $\eta_t=\frac{\partial \theta_t^\e}{\partial \e}\Big|_{\e=0}\equiv 0$ we obtain
\begin{align*}
\frac{D}{\partial \e}\Big|_{\e=0}\left(\frac{\partial X_t^\e}{\partial \e}\right)
=\frac{D}{\partial \e}\Big|_{\e=0}\left(U_t^\e \theta_t^\e\right)=
U_t\left(\w_t^0\theta_t^0+\frac{\partial \theta_t^\e}{\partial \e}\Big|_{\e=0}\right)
=U_t\Gamma_th(t).
\end{align*}
Now we have obtained \eqref{l3-3-1}. \end{proof}

\subsection{Proof for the 2nd order gradient formula on a compact manifold}
\begin{prp}\label{prp4.5}
Let $t>0$, $x\in M$ and $v\in T_x M$.
Then for any $f\in C_b(M)$ and $h\in L^{2,1}(\Omega;\R^n)$ satisfying that $h(0)=U_0^{-1}v$ and $h(t)=0~a.s.$,
we have
\begin{equation}\label{p3-1-0}
\big\langle \nabla P_tf(x), v\big\rangle_{T_x M}=-\EE\left[ f(X_t^x)\int_0^t \<\Theta_s^h, dB_s\>\right],
\end{equation}
where $ \Theta^h_t:=h'(t)+\frac{1}{2}{\rm ric}_{U_t}(h(t))$. Furthermore,
\begin{equation}\label{p3-1-1}
\begin{split}
&\quad \big\langle \nabla^2 P_t f(x), v\otimes v\big\rangle_{T_x M \otimes T_x M}\\
&=\EE\left[f(X_t^x)\left(
\left(\int_0^t\big\langle \Theta_s^h, \d B_s\big\rangle\right)^2-
\int_0^t \big\langle \Lambda_s^h, \d B_s\big\rangle-\int_0^t \left|\Theta_s^h\right|^2 \d s\right)\right].
\end{split}
\end{equation}
\end{prp}
\begin{proof}
We take $\xi(\cdot)$ to be a geodesic with initial value $\xi(0)=x$ and initial velocity $\frac{\partial \xi(\e)}{\partial \e}\Big|_{\e=0}=v$. Let
$\{U_0^\e\in (-1,1)\}$ denote the parallel orthonormal frame bundle along $\xi(\cdot)$ with $U_0^\e\Big|_{\e=0}=U_0$. In particular, it holds
that $\pi(U_0^\e)=\xi(\e)$. Recall that $U_t^\e$ is the solution to
\eqref{sde2} with initial value $U_0^\e$ chosen above.
It holds that
\begin{align*}
\int_0^t G_s^\e \circ dB_s^\e&=
\int_0^t G_s^\e \circ \d B_s+\int_0^t G_s^\e \left(\e h'(s)+\frac{\e^2}{2}\Gamma_s\,h'(s)\right)\d s\\
&=\int_0^t G_s^\e  \d B_s+\int_0^t \frac{1}{2}\d \langle G_{\cdot}^\e, B_{\cdot}\rangle_s+ \int_0^t G_s^\e \left(\e h'(s)+\frac{\e^2}{2}
\Gamma_s\,h'(s)\right)\d s\\
&=\int_0^t G_s^\e  \d B_s+\e\int_0^t G_s^\e \Theta_s \d s +\frac{\e^2}{2}\int_0^t G_s^\e \Lambda_s \d s.
\end{align*}
Here we have used that
\begin{align*}
\d \langle G_{\cdot}^\e, B_{\cdot}\rangle_t&=
-\e G_t^\varepsilon \d\langle\Gamma_\cdot , B_{\cdot}\rangle_t-\frac{\e^2}{2}
G_t^\varepsilon \d\langle\Gamma_\cdot^{(2)} , B_{\cdot}\rangle_t\\
&=\frac{\e}{2}G_t^\e\;\ric_{U_t}(h(t))\d t+\frac{\e^2}{2}G_t^\e\Big[U_t^{-1}\nabla {\rm Ric}^\sharp_{\pi(U_t)}\left(U_th(t),
U_t h(t)\right)\\
&\ +\Gamma_t \; \ric_{U_t}(h(t))-\; \ric_{U_t}\left(\Gamma_t h(t)\right)\Big].
\end{align*}
Note that $W_t^\e:=\int_0^t G_s^\e \d B_s$ is still an $\R^n$-valued Brownian motion, so we have
\begin{align*}
dU_t^\e&=H(U_t^\e)\circ \left(\d W_t^\e+G_s^\e\left(\e\Theta_t+\frac{\e^2}{2}\Lambda_t\right)\d t\right).
\end{align*}
Let
\begin{align*}
M_t^\e:=\exp\left(-\int_0^t \left\langle \e\Theta_s+\frac{\e^2}{2}\Lambda_s, \d B_s\right\rangle-
\int_0^t \left(\frac{\e^2}{2}\left|\Theta_s+\frac{\e}{2}\Lambda_s\right|^2\right) \d s \right).
\end{align*}
Then by the Girsanov theorem, the distribution of
$\{U_{s}^\e; s\in [0,t]\}$ under $d\mathbb{Q}^\e:=M_t^\e d\mathbb{P}$ is the same as that
of $\{U_s^{0,\e}; s\in [0,t]\}$, where  $U_\cdot^{0,\e}$ is the solution to equation \eqref{sde1} with initial value $U_0^{0,\e}=U_0^\e$.
Therefore we obtain
\begin{equation}\label{p3-1-1a}
P_t f(\xi(\e))=\EE\left[f(X_t^{\xi(\e)})\right]=\EE\left[f(X_t^{\e,\xi(\e)})M_t^\e\right],
\end{equation}
where $X_t^{\e,\xi(\e)}=\pi(U_t^\e)$, $X_t^{\xi(\e)}=\pi(U_t^{0,\e})$.

We first assume $f\in C_b^2(M)$, differentiating \eqref{p3-1-1a} with respect to $\e$ yields that
\begin{equation}\label{p3-1-2b}
\begin{split}
\big\langle\nabla P_t f (x), v\big\rangle_{T_x M}&=\frac{\partial}{\partial \e}\bigg|_{\e=0} P_t f\big(\xi(\e)\big)\\
&=\EE\left[\frac{\partial }{\partial \e}\bigg|_{\e=0}f\big(X_t^{\e,\xi(\e)}\big)\right]+\EE\left[f(X_t)\left(\frac{\partial }{\partial \e}\bigg|_{\e=0} M_t^\e\right)\right],
\end{split}
\end{equation}
Another round of differentiation gives:
\begin{equation}\label{p3-1-3}
\aligned
&\left\langle \nabla^2 P_t f(x), v\otimes v\right\rangle_{T_x M \otimes T_x M}=\frac{\partial^2}{\partial \e^2}\bigg|_{\e=0} P_t f(\xi(\e))
\\&=\EE\left[\frac{\partial^2 }{\partial \e^2}\bigg|_{\e=0}f\big(X_t^{\e,\xi(\e)}\big)\right]+2
\EE\left[\left(\frac{\partial }{\partial \e}\bigg|_{\e=0}f\big(X_t^{\e,\xi(\e)}\big)\right)
\left(\frac{\partial }{\partial \e}\bigg|_{\e=0} M_t^\e\right)\right]\\&+\EE\left[f\big(X_t^x\big)\frac{\partial^2 }{\partial \e^2}\bigg|_{\e=0}M_t^\e\right].
\endaligned
\end{equation}
According to the last line of \eqref{l3-1-2}, \eqref{l3-3-1}, the definition of $M_t^\e$ and the fact that $h(t)\equiv 0$ we derive
$$\frac{\partial }{\partial \e}\bigg|_{\e=0}f\Big(X_t^{\e,\xi(\e)}\Big)=\langle \nabla f(X_t^x),U_th(t)\rangle_{T_{X_t^x}M}=0$$
and also,
$$ \frac{\partial }{\partial \e}\bigg|_{\e=0} M_t^\e=-\int_0^t\langle \Theta_s, \d B_s\rangle.$$
Furthermore,
\begin{align*}
&\frac{\partial^2 }{\partial \e^2}\bigg|_{\e=0}f\Big(X_t^{\e,\xi(\e)}\Big)=\left \langle \nabla^2 f\big(X_t^x\big), \frac{\partial X_t^{\e,\xi(\e)}}{\partial \e}\bigg|_{\e=0}\bigotimes\frac{\partial X_t^{\e,\xi(\e)}}{\partial \e}\bigg|_{\e=0}\right\rangle_{T_{X_t^x}M\otimes T_{X_t^x}M}\\&\quad\quad\quad\quad\quad\quad\quad\quad \quad\quad +\left\langle \nabla f\big(X_t^x\big), \frac{D}{\partial \e}\bigg|_{\e=0}\left(\frac{\partial X_t^{\e,\xi(\e)}}{\partial \e}\right)\right\rangle_{T_{X_t^x}M}\\
&\quad\quad=\left\langle\nabla^2 f(X_t^x), U_t h(t)\otimes U_t h(t)\right\rangle_{T_{X_t^x}M\otimes
T_{X_t^x}M}+\left\langle \nabla f\big(X_t^x\big),U_t\Gamma_t h(t)\right\rangle_{T_{X_t^x}M}=0,\\
&\frac{\partial^2 }{\partial \e^2}\Big|_{\e=0} M_t^\e=\left(\int_0^t\langle \Theta_s, \d B_s\rangle\right)^2-
\int_0^t \langle \Lambda_s, \d B_s\rangle-\int_0^t \left|\Theta_s\right|^2 \d s.
\end{align*}
Crucially this special choice of variation ensures that $\frac{\partial^2 }{\partial \e^2}\bigg|_{\e=0}f\Big(X_t^{\e,\xi(\e)}\Big)$  depends only on $h(t)$, not on the history of the process $h$.

Putting these  back to \eqref{p3-1-2b} and \eqref{p3-1-3} yields \eqref{p3-1-0}, \eqref{p3-1-1} for $f\in C_b^2(M)$.
By standard approximation procedure and the compact property of $M$ we  see that these equalities still hold for any $f\in C_b(M)$.
\end{proof}

\section{Quantitative Cut-off Processes}\label{s2-1}
\label{cut-off}
\emph{From now on, we assume that $M$ is an $n$-dimensional general complete Riemannian manifold, not necessarily compact.}

In this section we introduce a class of cut-off processes satisfying  estimates
 crucial for the localisation procedures, which we shall apply later to (\ref{p3-1-1})  and to
obtain the asymptotic gradient estimates for the logarithmic heat kernel.

Since geodesic balls have typically non-regular boundary, we firstly construct a family of relatively compact open sets  $\{D_m\}_{m=1}^\infty$ with smooth boundary which plays the roles of geodesic balls and such that $\cup_{m=1}^\infty D_m=M$. Our localisation procedure crucially relies on $ D_m$  has smooth boundaries, see Lemma \ref{lem7.1}.
We first use a result in
Greene and Wu \cite{GW}
on the existence of a smooth approximate distance function,  which is valid for complete manifold, and then construct a family of cut off vector fields adapted to $\{D_m\}_{m=1}^\infty$. Fixing an $o\in M$, denote by $d$ the Riemannian distance function on $M$ from $o$.  Since $M$ is complete,
according to  \cite{GW} there exists a non-negative smooth function $\hat d:M\rightarrow\R_+$ with the property that $0<|\nabla \hat d|\le 1$  and
$$\left|\hat d(x)-\frac{1}{2}d(x)\right|<1,\quad \forall\; x\in M.$$
For every non-negative $m$, define $D_m:=\hat d^{-1}((-\infty,m)):=\{z \in M; \hat d(z)<m\}$,
then it is easy to verify
$B_o(2m-2)\subset D_m\subset  B_o(2m+2)$, where $B_o(r):=\{z \in M; d(z)<r\}$ is the geodesic ball centred at $o$ with radius $r$. Let $\phi:\R\rightarrow [0,1]$ be a smooth function such that
\begin{equation}\label{eq2.1}
\phi(r)=\left\{ \begin{array}{ll} 1, \qquad &r\leq1\\
\in (0,1), &r\in (1,2)\\
0, &r\geq2.\end{array}\right.
\end{equation}
Setting
\begin{equation}\label{eq2.2}
f_m(z):=\phi\Big(\hat d(z)-m+2\Big),\quad z\in M,
\end{equation}
then it is easy to see that
\begin{equation*}
f_m(z)=\left\{ \begin{array}{ll} 1, \qquad &\text{if}~~z \in \overline D_{m-1}\\
0, \qquad &\text{if}~~z \in D_m^c\\
\in (0,1), \qquad &\text{orthewise}\end{array}\right.
\end{equation*}
and
$D_m=\{z \in M; f_m(z)>0\}$. Without loss of generality we can assume that $D_m$ is a bounded connected open set
(otherwise we could take the connected component of $D_m$ containing $B_o(2m-2)$). Moreover, since
$\partial D_m=\{z\in M; \hat d (z)=m\}$ and $|\nabla \hat d(z)|\neq 0$ for all $z\in M$, we know $\partial D_m$ is a smooth
$n-1$ dimensional submanifold of $M$.

As before we suppose that $\{U_t\}_{0\le t<\zeta(x)}$ is the solution to the canonical horizontal equation \eqref{sde1}
with $\zeta(x)$ denoting its explosion time, and $\{X_t^x:=\pi(U_t)\}_{0\le t<\zeta(x)}$ is a Brownian motion
on $M$ with initial value $x:=\pi(U_0)$.

Let  $\partial$ denote the cemetery state for $M$ and set $\bar M=M\cup \{\partial\}$. Given a $x\in M$ we let
$$P_x(\bar M):=\{\gamma\in C([0,1];\bar M): \; \gamma(0)=x\}$$ denote
 the collection of all $\bar M$-valued continuous paths with initial vale $x$. Let
 $\mu_x$ denote the Brownian motion measure on $P_x(\bar M)$.
We also  refer  the natural filtration of the canonical process $\gamma(\cdot)$ as the canonical filtration on $P_x(\bar M)$,
which is augmented to be complete and right continuous as usual.

It is well known that the distribution of $\{X_{t}^x\}_{0\le t<\zeta(x)}$ and $\{U_t\}_{0\le t<\zeta(x)}$ under
$\Pp$ is the same as that of the canonical process $\{\gamma(t)\}_{0\le t<\zeta(\gamma)}$ and its horizontal lift
$\{U_t(\gamma)\}_{0\le t<\zeta(\gamma)}$ under $\mu_x$, where $\zeta(\gamma)$ denotes the explosion time
of  $\gamma(\cdot)$.  Set
 $$\tau_m(\gamma)=\tau_{D_m}(\gamma):=\inf\left\{s \ge 0:\ \gamma(s) \notin D_m\right\}.$$

\begin{lem}\label{lem5.1}
For any $m \in \mathbb{N}$ there exists a stochastic process (vector field)
$l_m: [0,1]\times P_x(\bar M) \rightarrow [0,1]$, such that
\begin{enumerate}

\item [(1)]  $l_m(t,\gamma)=\left\{ \begin{array}{ll} 1, \qquad &t \le \tau_{m-1}(\gamma)\wedge 1\\
0, &t > \tau_{m}(\gamma)\end{array}\right..$

\item[(2)]   {\bf Absolute continuity:}
$l_m(t,\cdot)$ is adapted to the canonical filtration and
$l_m(\cdot,\gamma)$ is absolutely continuous for $\mu_x$-a.s. $\gamma \in P_x(\bar M)$.

\item[(3)]  {\bf Local uniform moment estimates:}
For every positive integer $k \in \N $, we have
\begin{equation}\label{eq2.4}
\sup_{x \in D_{m-1}}\int_{P_x(\bar M)} \int^1_0|l_m'(s, \gamma)|^k \d s\; \mu_x(\d\gamma)\le C_1(m,k)
\end{equation}
for some positive constant $C_1(m,k)$ (which may depends on $m$ and $k$).
\end{enumerate}

\end{lem}

\begin{proof}
In the proof, the constant $C$ (which may depend on $m$) will change in different lines.
The main idea of the proof is inspired by the article of Thalmaier \cite{T} and Thalmaier and Wang \cite{TW}.

$(1)$ Since for any $m\geq1$, $D_m\subset D_{m+1}\uparrow M$, there exists a $m_0\in \mathbb{N}$ such that

$$~~~~~~~~\left\{ \begin{array}{ll} x\in D_m, \qquad &\text{when}~m\geq m_0\\
x\notin D_m, &\text{when}~1\leq m<m_0\end{array}\right..$$
When $x\notin D_m$, let $l_m(t,\gamma)\equiv0$.
In the following, we will consider the case of $x\in D_m$ (which implies that $\tau_m(\gamma)>0$) without loss of generality.
Let $f_m:M \rightarrow \R_+$ be the function given by \eqref{eq2.2}, we define a sequence of functions:
$$T_m(t,\gamma):=\left\{ \begin{aligned}\int^t_0\f {\d s}{ \left[ f_m\left(\gamma(s)\right)\right]^2},\quad\quad\quad\quad t< \tau_m(\gamma)\\
\infty, \quad\quad\quad\quad \quad\quad\quad t\ge\tau_m(\gamma).\end{aligned}\right. $$
Then each $T_m(\cdot, \gamma)$ is an increasing right continuous function of $t$.
For any $t\ge 0$, set
$$A_m(t,\gamma):=\inf\left\{s \ge 0:\ T_m(s,\gamma)\ge t\right\}.$$
We may omit the parameter $\gamma$ in the notation of $T_m(t,\gamma)$, $A_m(t,\gamma)$ for simplicity in the proof.

Since $\inf_{s\in [0,t]}f_m(\gamma(t))>0$ for $t<\tau_m(\gamma)$, then $T_m(t)<\infty$ for every
$t<\tau_m$ and $T_m(\cdot)$ is strictly increasing and continuous in $[0,\tau_m)$ (with respect to the variable $t$). Therefore $A_m(\cdot)$ is continuous on $[0,T_m(\tau_m))$ and $T_m(A_m(t))=t$ for every $0\le \tau_m<T_m(\tau_m)$.
Furthermore we have $T_m(\tau_m)=\infty$. To see this we only need to observe that
\begin{align*}
f_m(\gamma(s))&=f_m(\gamma(s))-f_m\left(\gamma(\tau_m)\right)
\le \frac{1}{2}\sup_{x\in D_m}|\nabla^2 f_m(x)|d\left(\gamma(s),\gamma\left(\tau_m\right)\right)^2\\
&\le C_m(\gamma)\sqrt{|s-\tau_m|},\qquad  \forall\ s<\tau_m,
\end{align*}
where $C_m(\gamma)$ is a constant,  and we applied the property that
$d\left(\gamma(s),\gamma\left(\tau_m\right)\right)\le C_m(\gamma)|s-\tau_m|^{1/4}$ which is easy to
prove by the Kolmogorov criterion. Combing the fact $T_m(\tau_m)=\infty$ with $T_m(t)<\infty$ for all
$0\le t<\tau_m$ immediately yields that $A_m(T_m(t))=t$ for every $0\le t\le \tau_m$ and $\tau_m>A_m(t)$ for every
$0\le t<\infty$.

Next,
we use the truncation function $\phi:\R \rightarrow \R$ in \eqref{eq2.1} to define
\begin{equation}\label{eq2.5}
l_m(t,\gamma)=\phi\left(\int_0^{t}   \f {\phi\big(T_m(s)-2\big)}{f_m^{2}\left(\gamma(s)\right)}\d s\right),
\end{equation}
which is clearly adapted to the canonical filtration.
Suppose that $t\geq \tau_m>A_m(3)$, then
\begin{equation*}
\begin{split}
\int_0^t \f {\phi\big(T_m(s)-2\big)}{f_m^{2}\left(\gamma(s)\right)}\d s
&\geq \int_0^t \1_{\{T_m(s)\le 3\}}f_m^{-2}\left(\gamma(s)\right)\d s\\
& =\int_0^t \1_{\{s\le A_m(3)\}}f_m^{-2}\left(\gamma(s)\right)\d s\\
&=T_m(A_m(3))=3,
\end{split}
\end{equation*}
which implies $l_m(t,\gamma)=0$ for $t \geq \tau_m$ by the definition of $\phi$.

If $s\le \tau_{m-1}(\gamma)$ then $f_m(\gamma(s))=1$ and so $T_m(s)=s$. Consequently,
$\phi\big(T_m(s)-2\big)=1$
for every $s\le \tau_{m-1}\wedge 1$. Hence we obtain
$$l_m(t,\gamma)=\phi( t\wedge 1)=1,\qquad \ \forall\;  t\le \tau_{m-1}\wedge 1, $$
concluding the proof of part (1).

$(2)$
Still by the expression of \eqref{eq2.5} we know the conclusion of part (2) holds.

$(3)$ Now it only remains  to verify the estimates \eqref{eq2.4}. Firstly,
\begin{equation*}
\begin{split}
|l_m'(t)|&=\left|\phi'\left(\int_0^t \f {\phi\big(T_m(s)-2\big)}
{f_m^{2}\left(\gamma(s)\right)}\d s\right)\right|
\f {\phi\big(T_m(t)-2\big)}
{f_m^{2}\left(\gamma(t)\right)}\\
&\leq \|\phi'\|_\infty f_m^{-2}\left(\gamma(t)\right)\1_{\{\phi(T_m(t)-2) \neq 0\}}\leq C  f_m^{-2}\left(\gamma(t)\right)\1_{\{T_m(t)\leq 4\}}.
\end{split}
\end{equation*}
Then for every $k \in \N $,
\begin{equation}\label{p3-1-2a}
\begin{split}
\int_0^1|l_m'(s)|^k\,\d s&\leq
C\int_0^1 f_m^{-2k}\left(\gamma(s)\right)\1_{\{T_m(s)\leq 4\}}\d s\\
&\leq C\int^{1}_0 f_m^{-2k+2}\bigl(\gamma(s) \bigl)
\1_{\{s\leq A_m(4)\}}\d T_m(s)\\
&=C\int^{4\wedge T_m(1)}_0 f_m^{-2k+2}\left(\gamma\left(A_m(r)\right)\right) \d r\\
&\le C\int^{4}_0 f_m^{-2k+2}\left(\gamma\left(A_m(r)\right)\right) \d r.
\end{split}
\end{equation}
Observe that the distribution of $X_{\cdot}^x$ under $\Pp$ is the same as that of $\gamma(\cdot)$ under $\mu_x$,
\begin{equation}\label{eq2.6}
\begin{split}
&\sup_{x \in D_{m-1}}\int_{P_x(\bar M)}\int^{4}_0 f_m^{-2k+2}\left(\gamma\left(A_m(s)\right)\right) \d s
\mu_x(\d \gamma)\\
&=\sup_{x \in D_{m-1}}\Ee\left[\int_0^{4} f_m^{-2k+2}\left(X_{A_m(s,X_{\cdot}^x)}^x\right)\d s\right].
\end{split}
\end{equation}
Let $S_{j,m}(\gamma):=\inf\left\{t>0; f_m\left(\gamma(t)\right)\le \frac{1}{j}\right\}$. According to It\^o's formula we obtain for all $j,k \in \mathbb{N}$ and $x\in D_{m-1}$,
\begin{equation}\label{p3-1-2}
\begin{split}
\EE \left[f_m^{-k} \left(X_{A_m(t)\wedge S_{j,m}} \right)\right]
&=f_m^{-k}(x) +\frac{1}{2}
\EE \left[\int_0^{A_m(t)\wedge S_{j,m}} \Delta \left(f_m^{-k}\right) (X_s) \d s\right]\\
&=1 +
\frac{1}{2}\EE\left[\int_0^{A_m(t)\wedge S_{j,m}} \left(f_m^2 \Delta \left(f_m^{-k}\right)\right) \left(X_{A_m(T_m(s))}\right) \d T_m(s)\right],
\end{split}
\end{equation}
we have applied the fact
that ${A_m\left(T_m(s)\right)=s}$ for every $0\le s<S_{j,m}$ and $f_m(x)=1$ for all $x\in D_{m-1}$.
Meanwhile we have
\begin{equation*}
\begin{split}
f_m^2\Delta(f_m^{-k})&=k(k+1)f_m^{-k}|\nabla f_m|^2-k f_m^{-k+1}\Delta f_m\\
&=k(k+1)f_m^{-k}\left|\phi'\left(\hat d-m+2\right)\right|^2
\left|\nabla \hat d\right|^2\\
&-k f_m^{-k}\left(f_m\phi''\left(\hat d-m+2\right)\left|\nabla \hat d\right|^2+\phi'
\left(\hat d-m+2\right)f_m\Delta \hat d\right)\\
&\leq k(k+1)f_m^{-k}\left(\|\phi'\|_{\infty}+\|\phi''\|_{\infty}+
\|\phi'\|_{\infty}\sup_{z \in D_m}|\Delta \hat d(z)|\right)\\
&\le Cf_m^{-k}.
\end{split}
\end{equation*}
Putting this into \eqref{p3-1-2} we arrive at
\begin{align*}
\EE \left[f_m^{-k} \left(X_{A_m(t)\wedge S_{j,m}} \right)\right]&
\le 1 +C\EE\left[\int_0^{A_m(t)\wedge S_{j,m}}f_m^{-k}\bigl(X_{A_m(T_m(s))}\bigr)\d T_m(s)\right]\\
&\le 1+C\int_0^t \EE\left[f_m^{-k}\bigl(X_{A_m(r)\wedge S_{j,m}}\bigr)\right]\d r,
\end{align*}
where the last step follows from the procedure of change of variable $u=T_m(s)$ and the fact $A_m(t)\le t$.

Hence by Grownwall's  inequality we arrive at for all $k,j\in \mathbb{N}$,
\begin{align*}
\EE \left[f_m^{-k} \bigl(X_{A_m(t)\wedge S_{j,m}} \bigr)\right]\le Ce^{Ct}.
\end{align*}
Then letting $j \rightarrow \infty$ and observing that $A_m(t)\le \tau_m=\lim_{j \rightarrow \infty}S_{j,m}$ we obtain for
all $k\in \mathbb{N}$,
\begin{align*}
\EE \left[f_m^{-k} \left(X_{A_m(t)} \right)\right]\le Ce^{Ct},
\end{align*}
combing this with \eqref{p3-1-2a}  yields \eqref{eq2.4}. This completes the proof for Lemma \ref{lem5.1}.
\end{proof}

\section{Proof of the Main Estimates}\label{section5}
In this section, we shall apply the cut-off
procedures, using the quantitative localised vector fields introduced in Section \ref{cut-off}, to obtain short time as well as asymptotic
first and second order gradient estimates for the logarithmic heat kernel of a complete Riemannian manifold without imposing on it any curvature bounds.

Let $\{D_m\}_{m=1}^{\infty}$ and $\{f_m\}_{m=1}^{\infty}$ be the sequences of domains and functions constructed in Section
\ref{cut-off}. Recall that for every $m$, $D_m=\{x \in M: f_m(x)>0\}$ is a bounded connected open set.
By Lemma \ref{lem7.1} from the Appendix there exists a compact Riemannian manifold $\tilde M_m$ such that $D_m$ is isometrically embedded into
$\tilde M_m$ as an open set. We could and will view $D_m \subset \tilde M_m$ as an open subset of $\tilde M_m$.
In particular, we have
\begin{equation}\label{eq3.5}
d_{\tilde M_m}(x,y)=d(x,y),\quad \quad  \forall\ x,y \in B_{o}(2m-2),
\end{equation}
where $d$ and $d_{\tilde M_m}$ are the Riemannian distance function on $M$ and $\tilde M_m$.
We denote the heat kernel on $M$ and $\tilde M_m$ by $p(t,x,y)$ and $p_{\tilde M_m}(t,x,y)$ respectively.
For every $e\in \R^n$ we also let $H_{e}^m$ denote the horizontal lift of $ue$ on $TO(\tilde M_m)$.

Let us fix a probability space $(\Omega, \mathscr{F},\Pp)$. Let  $\{B_t\}_{t\ge 0}$ be the standard $\R^n$-valued Brownian motion with
$B_t=(B_t^1, \dots, B_t^n)$, and we denote by $\mathscr{F}_t$ the filtration generated by it. Now we fix an orthonormal basis
$\{e_i\}_{i=1}^n$  of $\R^n$.

For  $x\in D_{m}\subset \tilde M_m$ and $U_0$ a frame  at $x$ so that $U_0 \in O_xM=O_x\tilde M_m$,
 let $U_t^m$ denote the solution to the following
$O(\tilde M_m)$-valued stochastic differential equation
\begin{equation}\label{e4-1}
\d U_t^m=\sum_{i=1}^n H_{e_i}^m\left(U_t^m\right)\circ \d B_t^i,\quad \ U_0^m=U_0.
\end{equation}

 Set $X_t^{m,x}:=\pi(U_t^m)$. This is a $\tilde M_m$-valued Brownian motion. Recall that $X_t^x:=\pi(U_t)$, where $U_t$ is the solution to \eqref{sde1},
 with the same driving Brownian motion $B_t$ and the same initial value $U_0$ as in \eqref{e4-1}.

 Throughout this section,  for every $m,k\in \mathbb{N}$ with $k\ge m$, we define
$$\tau_m:=\inf\{t>0; X_t^x\notin D_m\},\quad  \quad \tau_{m}^k:=\inf\{t>0; X_t^{k,x}\notin D_m\}.$$
Note that for every $k>m$, $H_{e_i}^k=H_{e_i}$ on $\pi^{-1}(D_m)$.   It is easy to verify that 
\begin{equation}\label{e4-2}
\begin{aligned}\tau_m=& \tau_{m}^k, \quad
   X_t^x=& X_t^{k,x}, \qquad \forall\ k\ge m>1,\ 0\le t\le \tau_m.
   \end{aligned}
\end{equation}
As before, the superscript $x$ may be omitted from time to time when there is no risk of confusion. \emph{The probability and the expectation
for the functional generated by $X_{\cdot}^x$ or $X_{\cdot}^{m,x}$ (with respect to $\Pp$) are  denoted by $\mathbb{P}_x$ and
$\Ee_x$ respectively in this section.}

\emph{If $M$ is compact, then when $m$ is large enough we have $D_m=M$ and we can take $\tilde M_m=M$
(we do not have to apply Lemma \ref{lem7.1} when $M$ is compact), then all the conclusions in this section
automatically hold. Hence
in this section we always assume that $M$ is non-compact.}

\medskip

We shall use the following estimates which are crucial for our proof.

\begin{lem}[\cite{Azencott, Molchanov, Va1,Va2}]\label{lem6.1}
For any $x, y \in M$,
\begin{equation}\label{eq3.1}
\lim_{t \downarrow 0}t\log p(t,x,y)=-\frac{d(x,y)^2}{2}.
\end{equation}
and the convergence is uniformly in $(x,y)$ on $ K\times K$ for any compact subset $K$.

Moreover, for every connected bounded open set $D\supseteq K$ with smooth boundary,
\begin{equation}\label{eq3.2}
\lim_{t \downarrow 0}t\log \Pp_x\left(\tau_D<t\right)=-\frac{d(x,\partial D)^2}{2},\quad \ \forall\  x\in K.
\end{equation}
Here  $\tau_D:=\inf\{t>0; X_t \notin D\}$ is the first exit time from $D$ and $d(x,\partial D):=\inf_{z \in \partial D}d(x,z)$.
And the convergence is also uniform in $x$ on $K$.
\end{lem}
The asymptotic estimates \eqref{eq3.1} and \eqref{eq3.2} were firstly shown to hold for $\R^n$ in  Varadhan
 \cite{Va1,Va2}, extension to a complete Riemannian manifold was given in
 Molchanov \cite{Molchanov}.
In addition,
Azencott \cite{Azencott-as} and  \cite{Hsu90} indicated that these statements may fail for an
 incomplete Riemannian manifold.  
 We shall also use the following statement, which follows readily  from the small time asymptotics and the Gaussian heat kernel upper bounds.
\begin{lem}( \cite{Azencott-as}, \cite[Lemma 2.2]{Hsu90})\label{lem6.2}
For any compact subset  $K$ of $M$ and any positive number~$r$,
Then there exists a positive number $t_0$ such that
\begin{equation}\label{eq3.3}
\sup_{t\in (0,t_0]} \sup_{ d(z,y) \ge r, \;y \in K}p(t,z,y)\le 1.
\end{equation}
\end{lem}

\subsection{Comparison theorem for functional integrals involving approximate heat kernels}
Let $D_m$ dentoe the relatively compact subset and let $l_m:[0,1]\times P_x(M)\rightarrow \R$ be the cut-off processes
adapted to $D_m$, as constructed by Lemma \ref{lem5.1}. Let $p^{D_m}(t,x,y)$ denote the Dirichlet heat kernel on $D_m$.
Let $K$ be a compact set and $x,y\in K$ be such that $d(x,y)<d(x, \partial D^m)\vee d(y, \partial D^m)$, then $p(t,x,y)$ and $p^{D_m}(t,x,y)$ are asymptotically the same for small $t$. See \cite{Azencott-as}, Lemma 2.3 on page 156.

Below we give a quantitative estimate on $p$ and $p^{D_m}$ on a compact set $K\times K$, for sufficiently large $m$. By sufficiently large, we mean that  $m\ge m_0$ for
a natural number $m_0$ and $m_0$ may depend on other data. In all the results below, it depends on the compact set $K$ and the prescribed exponential factor $L>0$.

\begin{lem}\label{lem6.3}
Suppose that $K$ is a compact subset of $M$ and $L>1$ is a positive number.  Then for sufficiently large $m$, there exists a positive number
$t_0=t_0(K,L,m)$ such that  for every $t\in (0,t_0]$,
\begin{equation}\label{eq3.6}
\begin{aligned}
\sup_{x,y\in K}\left|p(t,x,y)-p^{D_{m}}(t,x,y)\right| &\le  \, e^{-\frac{2L}{t}}, \\\sup_{x,y\in K} \left|p_{\tilde M_{m}}(t,x,y)-p^{D_{m}}(t,x,y)\right| &\le e^{-\frac{2L}{t}}.
\end{aligned}
\end{equation}
 In particular, for every $t\in (0,t_0]$,   \begin{equation}\label{eq3.7a}
 \sup_{x,y\in K}\left|p(t,x,y)-p_{\tilde M_{m}}(t,x,y)\right| \leq e^{-\frac{L}{t}}.
 \end{equation}
\end{lem}

\begin{proof}
The estimates in \eqref{eq3.6} could be found in Azencott  \cite[section 4.2]{Azencott-as}, \and also in Bismut \cite[Section III.a]{Bis} and  Hsu \cite[The proof of Theorem 5.1.1]{Hsu2}. Here
we  include a proof for the convenience of the reader. The technique and the intermediate estimates will be used  later.

By the strong Markovian property,
$$P_tf(x)=\Ee_x \left[f\left(X_t\right)\1_{\{t\le \tau_{m}\}}\right]+\Ee_x \left[\Ee_{X_{\tau_m}}
\left[f\left(X_{t-\tau_m}\right)\right]\1_{\{\tau_{m}<t<\zeta\}}\right]$$
 and so for any $x,y \in K$ and $ t>0$,
\begin{equation}\label{eq3.8}
\begin{split}
p(t,x,y)=p^{D_m}(t,x,y)+\Ee_x\left[ p\left(t-\tau_{m},X_{\tau_{m}},y\right)
\1_{\{\tau_{m}<t<\zeta\}}\right].
\end{split}
\end{equation}

Since $M$ is non-compact, given any number  $L>1$, there exists a natural number $m_0$ such that
$$K\subset B_o(2m_0-2), \quad  d(K,\partial D_{m_0})\ge d(K,\partial B_o(2m_0-2))>4L.$$
Then, according to \eqref{eq3.2} and \eqref{eq3.3}, for every $m\ge m_0$, we could find a positive number  $t_0(K,L,m)$ such that for any $t\in (0,t_0]$,
\begin{equation*}
\begin{split}
& \Pp_x\left(t>\tau_{m}\right)\le \exp\left(-\frac{d(x,\partial D_{m})^2-1}{2t}\right)\le
e^{-\frac{2L}{t}}, \qquad \forall x\in K,\\
& p(t,z,y)\le 1, \quad  \hbox{ for  all } \; z\in \partial D_m, \hbox{ and  } \; y\in K.\end{split}
\end{equation*}
By these estimates we obtain that, for all $m\ge m_0$ and  all $t \in (0,t_0]$,
\begin{equation*}
\begin{split}
\Ee_x\left[p\left(t-\tau_{m},X_{\tau_m},y\right)
\1_{\{t> \tau_{D_{m}}\}}\right]
&\le \sup_{t \in (0,t_0)}\sup_{ z\in \partial D_{m}, y \in K}p(t,z,y)\cdot
\Pp_x\left(t>\tau_{m}\right)\\
&\le e^{-\frac{2L}{t}}.
\end{split}
\end{equation*}
Putting this into \eqref{eq3.8} we arrive at that for all $m\ge m_0$, all  $x,y\in K$, and for all $ t \in (0,t_0]$,
\begin{equation}\label{eq3.9}
\begin{split}
& |p(t,x,y)-p^{D_{m}}(t,x,y)|\le e^{-\frac{2L}{t}}.
\end{split}
\end{equation}

Note that for every $m\ge m_0$, $D_{m}\subset \tilde M_{m}$ and $x\in K$,
$$d_{\tilde M_{m}}(x, \partial D_{m})\ge
d_{\tilde M_{m}}(x, \partial B_o(2m-2))
=d(x, \partial B_o(2m-2))$$
 which is due to \eqref{eq3.5}.
By the same argument for \eqref{eq3.9} and changing the constant $t_0$ if necessary we could find a
$t_0(K,L,m)$ such that for all $m\ge m_0$,
\begin{equation*}
\begin{split}
& \left|p_{\tilde M_{m}}(t,x,y)-p^{D_{m}}(t,x,y)\right|\le e^{-\frac{2L}{t}},\quad \ x,y\in K,\ t \in (0,t_0].
\end{split}
\end{equation*}
This, together with \eqref{eq3.9},  yields \eqref{eq3.6} and \eqref{eq3.7a}.
\end{proof}

\begin{lem}\label{lem6.4}
Suppose that $K$ is a compact subset of $M$ and $L>1$ is a positive number.
\begin{itemize}
\item[$(1)$] For $m_0$ sufficiently large and any $m>m_0$,
 there exists a $t_0(K,L,m)$ such that  for every  $0<s\le \frac{t}{2}$ and $0<t\le t_0$, we have
\begin{equation}\label{l4-1-7}
\sup_{x, y\in K}\sup_{z\in D_{m_0}} \left| \f {p(t-s, x,z)}{p(t,x,y)}- \f {p_{\tilde M_m}(t-s, x,z)}{p_{\tilde M_m}(t,x,y)} \right|
\le 2e^{-\frac{4L}{t}}.
 \end{equation}
\item[$(2)$]  Suppose $\Upsilon_t$ is an $\mathscr{F}_t$ adapted process, and for any $q>0$ and $m\geq1$ we set
$$F^q_m(\Upsilon, X_\cdot)=\Bigl( \int_0^s \Upsilon_r l_m'\Bigl(r,X_{\cdot}\Bigr)\d B_r\Bigr)^q,\quad
F_m^q(\Upsilon, X_\cdot^m)= \Bigl(\int_0^s \Upsilon_r l_m'\Bigl(r,X_{\cdot}^m\Bigr)\d B_r\Bigr)^q.$$
We also assume that
\begin{equation}\label{l4-1-0}
\sup_{x\in K}\Ee_x\left[\int_0^{1\wedge \tau_m} |\Upsilon_s|^{2q}\d s\right]<\infty, \quad \forall m\geq1
\end{equation}
for some $q\in \mathbb{N}$.
Then for every sufficiently large $m$ (any  $m$  greater than some number $m_0(K,L)$),  we can find a positive number
$t_0(K,L,m)$
with the property that
\begin{equation}\label{l4-1-1}
\begin{split}
&\sup_{x,y\in K}\Bigg |\Ee_x \Bigl[ F_m^q(\Upsilon, X_\cdot) \; \f {p(t-s, X_s,y)}{p(t,x,y)}\;
\Bigr]
-\EE_x\Bigl[F_m^q(\Upsilon, X_\cdot^m)\;
 \f {p_{\tilde M_m}(t-s, X_s^m,y)}{p_{\tilde M_m}(t,x,y)} \Bigr]\Bigg|\\
& \le  C(m)\;e^{-\frac{L}{t}}
\end{split}
\end{equation}
for any $0<t\le t_0$, $0<s\le \frac{t}{2}$.
Here the positive constant $C(m)$ may depend on $m$ and on
 $\alpha_m:=\sup_{x\in K}\Ee_x\left[\int_0^{1}\left|\Upsilon_r l_m'\left(r,X_{\cdot}\right)\right|^{q} \d r\right]$. (Note that $l_m'\left(r,X_{\cdot}\right)\neq 0$ only for $r<\tau_m=\tau_{m}^m$ so the quantity is well defined.)
\end{itemize}
\end{lem}
\begin{proof}
In the proof,  the constant $C$ may represent different constants in different lines.
Let $r_0:=\sup_{x,y\in K}d_M(x,y)$ denote the diameter of $K$. Since
$M$ is non-compact, we can choose a natural number $\tilde m_0$ (which may depend on $K$ and $L$) such that
$$K\subset B_o(2\tilde m_0-2)\subset D_{\tilde m_0}$$ and for all $m>\tilde m_0$,
$$d\left(K, \partial  B_o(2\tilde m_0-2)\right)=d_{\tilde M_m}
\left(K, \partial B_o(2\tilde m_0-2)\right)> 4(L+ r_0+1).$$

Also, by the heat kernel comparison \eqref{eq3.6} and \eqref{eq3.7a}, we can find a $m_0>\tilde m_0$ so that for all $m>m_0$, there exists a constant $t_2(K,L,m)>0$ such that,
\begin{equation}\label{eq3.12}
\begin{split}
\left|p(t,z,y)-p_{\tilde M_{m}}(t,z,y)\right| \le e^{-\frac{4(L+r_0+1)^2}{t}},\quad\quad \ \forall\ t\in (0,t_2],\ z,y\in D_{m_0}.
\end{split}
\end{equation}
According to the asymptotic relations \eqref{eq3.1} and \eqref{eq3.2},  for every $m>m_0$( taking $m_0$ larger as is necessary) we could find a constant $0<t_1(K,L,m)\le t_2$ such that  for all $t\in (0,t_1]$,
\begin{eqnarray}
\label{l4-1-3}
&&p(t,z,y)\le  e^{\frac{1}{t}}, \qquad \ \qquad  p_{\tilde M_m}(t,z,y)\le e^{\frac{1}{t}},\quad \qquad \quad \forall z,y\in D_{m_0},
\\
\label{l4-1-4}
&&p(t,z,y)\ge e^{-\frac{r_0^2+1}{t}},\qquad  \; p_{\tilde M_m}(t,z,y)\ge e^{-\frac{r_0^2+1}{t}},\; ~\qquad \forall  z,y\in K,
\\
\label{l4-1-5}
&&\Pp_z\left(\tau_{m_0}<t\right)\le e^{-\frac{4(L+r_0+1)^2}{t}},
\qquad\qquad\qquad\qquad \qquad ~~~\forall\,\ \ z\in K.
\end{eqnarray}
By the small time locally uniform heat kernel bound \eqref{eq3.3}, for every $m>m_0$ there exists a number $0<t_0(K,L,m)\le t_1$ such that
for all $ t\in (0,t_0]$,
\begin{equation}\label{l4-1-6}
p(t,z_1,y)\vee\ p_{\tilde M_m}(t,z_2,y)\le 1,   ~~~~~~\quad \ \forall  z_1\in M\cap D_{m_0}^c,\;
z_2\in \tilde M_m\cap D_{m_0}^c, \; y\in K.
\end{equation}
Therefore for every $m>m_0$ and for every  $0<s\le \frac{t}{2}$, every $0<t\le t_0$, and for all $x,y\in K$ and $ z\in D_{m_0}$, we have
 \begin{equation}
 \begin{split}
&\left| \f {p(t-s, x,z)}{p(t,x,y)}- \f {p_{\tilde M_m}(t-s, x,z)}{p_{\tilde M_m}(t,x,y)} \right| \\
&\le \frac{ p(t-s,x,z)\left|p_{\tilde M_m}(t,x,y) -p(t,x,y)\right|+
p(t,x,y) \left|p_{\tilde M_m}(t-s,x,z) -p(t-s,x,z) \right|  }{p(t,x,y)p_{\tilde M_m}(t,x,y)}\\
&\le  2e^{\frac{2(1+r_0^2)}{t}} e^{\frac{2}{t}} e^{-\frac{4(L+r_0+1)^2}{t}}
\le 2e^{-\frac{4L}{t}}.
 \end{split}
 \end{equation}
Here the second step above is due to \eqref{eq3.12}--\eqref{l4-1-4}. Thus, we finish the proof of $(1)$.

For all $m>m_0$, let us split the terms as follows,
\begin{align*}
&\Ee_x \left[ F^q_m(\Upsilon, X_\cdot) 
 \f {p(t-s, X_s,y)}{p(t,x,y)}
\right]\\
=&\Ee_x \left[ F^q_m(\Upsilon, X_\cdot)  \f {p(t-s, X_s,y)}{p(t,x,y)}\1_{\{t\le \tau_{m_0}\}}\right]
+\Ee_x \left[ F^q_m(\Upsilon, X_\cdot)  \f {p(t-s, X_s,y)}{p(t,x,y)}\1_{\{ t>\tau_{m_0}\}}\right]\\
=&:I_1^m(s,t)+I_2^m(s,t).
\end{align*}
Since $l_m'\left(r,X_{\cdot}\right)\neq 0$ if only if $t<\tau_m$, then $l_m'\left(r,X_{\cdot}^m\right)=l_m'\left(r, X_{\cdot}\right)$ and
we have
\begin{align*}
F_m^q(\Upsilon,X_\cdot^m)&=\left(\int_0^{s} \Upsilon_r l_m'\left(r,X^m_{\cdot}\right)\d B_r\right)^q=
\left(\int_0^{s\wedge \tau_m} \Upsilon_r l_m'\left(r,X^m_{\cdot}\right)\d B_r\right)^q\\
&=\left(\int_0^{s} \Upsilon_r l_m'\left(r,X_{\cdot}\right)\d B_r\right)^q=F_m^q(\Upsilon,X_\cdot).
\end{align*}
Note also,  $X_s^m=X_s$ for every $s\le \frac{t}{2}<\tau_m$.  It holds that
\begin{align*}
&\Ee_x \left[ F_m^q(\Upsilon,X_\cdot^m)
 \f {p_{\tilde M_m}(t-s, X_s^m,y)}{p_{\tilde M_m}(t,x,y)}\right]\\
=&\Ee_x \left[ F^q_m(\Upsilon, X_\cdot)  \f {p_{\tilde M_m}(t-s, X_s,y)}{p_{\tilde M_m}(t,x,y)}\1_{\{t\le \tau_{m_0}\}}\right]
+\Ee_x \left[ F^q_m(\Upsilon, X_\cdot)  \f {p_{\tilde M_m}(t-s, X_s^m,y)}{p_{\tilde M_m}(t,x,y)}\1_{\{t> \tau_{m_0}\}}\right]\\
=&:J_1^m(s,t)+J_2^m(s,t),\quad 0<s<\frac{t}{2},
\end{align*}
Note that
\begin{equation}\label{14-1-11}\alpha_m=\sup_{x\in K}\Ee_x\left[\int_0^1 \left|\Upsilon_r l_m'\left(r,X_{\cdot}\right)\right|^q \d r \right]<\infty.\end{equation}
This follows from the moment estimates on $l_m$, \eqref{eq2.4}, the assumption \eqref{l4-1-0}, and
also $$
\alpha_m \le
\sup_{x\in K}\Ee_x\left[\int_0^{1\wedge \tau_m} \left|\Upsilon_r\right|^{2q} \d r \right]^{1/2}
\sup_{x\in K}\Ee_x\left[\int_0^{1\wedge \tau_m} \left|l_m'\left(r,X_{\cdot}\right)\right|^{2q} \d r \right]^{1/2}.
$$

For all
$m>m_0$,  $x,y\in K$, $0<s\le \frac{t}{2}$, and  $0<t\le t_0$, we may assume that $t_0\le 2$,
\begin{align*}
&|I_1^m(s,t)-J_1^m(s,t)|\\
&\le \sup_{z\in D_{m_0}}\left| \f {p(t-s, z,y)}{p(t,x,y)}
- \f {p_{\tilde M_m}(t-s, z,y)}{p_{\tilde M_m}(t,x,y)} \right|
\Ee_x\left[\left|\int_0^{s} \Upsilon_r l_m'\left(r,X_{\cdot}\right)\d B_r\right|^q \right]\\
&\le  Ce^{-\frac{4L}{t}}\sup_{x\in K}\Ee_x\left[\int_0^{1\wedge \tau_m} \left|\Upsilon_r l_m'\left(r,X_{\cdot}\right)\right|^q \d r \right]=C \alpha_m\;e^{-\frac{4L}{t}}.
\end{align*}
In the penultimate step, we have applied  Burkholder-Davies-Gundy inequality and \eqref{l4-1-7}.
According to \eqref{l4-1-3} and \eqref{l4-1-6} we also have
$$\sup_{z\in M, y\in K}p(t,z,y)\le e^{\frac{1}{t}},\ \ \forall\ 0<t\le t_0.$$
Combining this with \eqref{l4-1-4}--\eqref{l4-1-5}, Cauchy-Schwartz inequality, and Burkholder-Davies-Gundy inequality, we obtain
 that for every $m>m_0$, $x,y\in K$, $0<s\le \frac{t}{2}$, and $0<t\le t_0$,
\begin{align*}
&|I_2^m(s,t)|\\
&\le C\; e^{\frac{r_0^2+1}{t}}\;\sup_{r\in [\frac{t}{2},t],z\in M,y\in K}p(r,z,y)\;
\Ee_x\left[\left|\int_0^{s} \Upsilon_r l_m'\left(r,X_{\cdot}\right)\d B_r\right|^{2q} \right]^{1/2}
\Pp_x\left(\tau_{m_0}<t\right)^{1/2}\\
&\le Ce^{\frac{r_0^2+1}{t}}e^{\frac{2}{t}}e^{-\frac{2(r_0+L+1)^2}{t}}\Ee_x\left[\int_0^{1\wedge \tau_m} \left|\Upsilon_r l_m'\left(r,X_{\cdot}\right)\right|^{2q} \d r \right]^{1/2}\le C\alpha_m \;e^{-\frac{2L}{t}}.
\end{align*}
Here in the last step we  used (\ref{14-1-11}).
Similarly, we obtain that for every $m>m_0$, $x,y\in K$,
\begin{align*}
|J_2^m(s,t)|\le C \alpha_m \;e^{-\frac{2L}{t}}, \qquad \hbox{ for all $0<s\le \frac{t}{2}$ and $0<t\le t_0$.}
\end{align*}
Combing the  above estimates for $I_1^m,I_2^m,J_1^m,J_2^m$ we see that,  for every
$m>m_0$, $x,y\in K$, $0<s\le \frac{t}{2}$ and $0<t\le t_0$,
\begin{align*}
&\Bigg |\Ee_x \left[ F_m^q(\Upsilon, X_\cdot) \f {p(t-s, X_s,y)}{p(t,x,y)}
\right]
-
 \EE_x\left[F_m^q(\Upsilon, X_\cdot^m)
 \f {p_{\tilde M_m}(t-s, X_s^m,y)}{p_{\tilde M_m}(t,x,y)} \right]\Bigg|\\
 &\le |I_1^m(s,t)-J_1^m(s,t)|+|I_2^m(s,t)|+|J_2^m(s,t)|\le C \alpha_m e^{-\frac{L}{t}},
\end{align*}
which is \eqref{l4-1-1} and we have finished the proof.
\end{proof}

\begin{remark}\label{rem6.1}
By the same arguments in the proof for \eqref{l4-1-1} we could obtain the following  under the conditions of Lemma
\ref{lem6.4}:  For sufficiently large $m$ we could find a positive number  $t_0(K,L,m)$ so that for every $x,y\in K$, $0<s\le \frac{t}{2}$, and  $0<t\le t_0$ the following estimates hold, replacing $l_m$ by $l_{m'}$, or $dB_r$ by $dr$.
\begin{equation}\label{r4-1-1}
\begin{split}
&\Bigg |\Ee_x \left[ \left(\int_0^{s} \Upsilon_r  \,l_m
(r,X_{\cdot})\d B_r\right)^q
\left(   \f {p(t-s, X_s,y)}{p(t,x,y)}
- \f {p_{\tilde M_m}(t-s, X_s^m,y)}{p_{\tilde M_m}(t,x,y)}\right)
\right]   \Bigg|
\le  C(m)e^{-\frac{L}{t}},
\end{split}
\end{equation}
\begin{equation}\label{r4-1-2}
\begin{split}
&\Bigg |\Ee_x \left[ \left(\int_0^{s} \Upsilon_r l_m'\left(r,X_{\cdot}\right)\d r\right)^q  \left( \f {p(t-s, X_s,y)}{p(t,x,y)}
-  \f {p_{\tilde M_m}(t-s, X_s^m,y)}{p_{\tilde M_m}(t,x,y)}  \right) \right]\Bigg|
\le  C(m)e^{-\frac{L}{t}},
\end{split}
\end{equation}
\begin{equation}\label{r4-1-3}
\begin{split}
&  \Bigg |\Ee_x \left[ \left(\int_0^{s} \Upsilon_r l_m\left(r,X_{\cdot}\right)\d r\right)^q
\left(  \f {p(t-s, X_s,y)}{p(t,x,y)}
-  \f {p_{\tilde M_m}(t-s, X_s^m,y)}{p_{\tilde M_m}(t,x,y)}\right)
\right]\Bigg|
\le  C(m)e^{-\frac{L}{t}}.
\end{split}
\end{equation}
These estimates will also be used in
the proof for Proposition \ref{prp6.6}.
\end{remark}

\subsection{Proof of the main theorem: gradient estimates}
\begin{lem}\label{lem6.5}
Let $t>0$, $x\in M$ and $v\in T_x M$. Suppose that $m$ is a natural number such that  $x\in D_m$. Let
 $h\in L^{1,2}(\Omega;\R^n)$ be given by
$$h(s)=\left(\frac{t-2s}{t}\right)^+\times l_m(s,X_{\cdot})\times U_0^{-1}v.$$
Then for any $f\in C_b(M)$ we have
\begin{equation}\label{l5-2-1}
\begin{split}
\langle \nabla P_tf(x), v\rangle_{T_x M}&=-\EE_x\left[ f(X_t)\int_0^t \<\Theta_s^h, dB_s\> \1_{\{t<\zeta\}}\right],
\end{split}
\end{equation}
where $\Theta_s^h$ defined by \eqref{e3-2} with the $h$ chosen above.
\end{lem}

\begin{proof}
When $M$ is compact, \eqref{l5-2-1} is just \eqref{p3-1-0} established in Proposition \ref{prp4.5}. For general
non-compact complete $M$, we will use the arguments based on truncation and approximation.
 For each $k>m$, let $\{U_t^k\}_{t\ge 0}$ be the horizontal Brownian motion on compact manifold $\tilde M_k$ as defined in (\ref{e4-1})
 with $\pi(U_0^k)=x\in D_m$.  Set $X_t^k=\pi(U_t^k)$ and $P_t^k f(x)=\Ee_x\left[f(X_t^k)\right]$.
Let
\begin{equation}\label{l5-2-1a}
h(s)=\left(\frac{t-2s}{t}\right)^+ \cdot l_m(s,X_\cdot)\cdot U_0^{-1}v.
\end{equation}

According to \eqref{e4-2}, precisely $\tau_m= \tau_{m}^k$ and $h(s)\neq 0$  if and only if when  $s\le \f t 2\wedge   \tau_m$.
So furthermore
$$h(s)=\left(\frac{t-2s}{t}\right)^+ \cdot l_m(s,X_\cdot^k)\1_{\{s<\f t 2 \wedge \tau_m\}}\cdot U_0^{-1}v,\quad \forall\ k>m,$$
$$h'(s)=\left(-\f 2 tl_m(s,X_\cdot^k) + l_m'(s,X_\cdot^k)\left(\frac{t-2s}{t}\right)^+\right)\1_{\{s<\f t 2 \wedge \tau_m\}}\cdot U_0^{-1}v,\quad \forall\ k>m,$$
which means that  we can replace $X_\cdot$ by $X_\cdot^k$ in the expression of $h(s)$ and $h'(s)$.
Let $\Theta_s^{h,k}$ be given by \eqref{e3-2}, with the manifold $M$ replaced by  $\tilde M_k$
(associated with $X_\cdot^k$). Therefore by \eqref{e3-2} we have the following expression
\begin{equation}\label{add1}
\Theta_s^{h,k}=h'(s)+\ric^{\tilde M_k}_{U_s^k}(h(s))=h'(s)+\ric_{U_s}(h(s))=\Theta_s^h,\quad \forall\ k>m.
\end{equation}
Here both sides of \eqref{add1} vanish for $s>\tau_m$, meanwhile we have used the fact that $U_s=U_s^k$ when $s<\tau_m$ and $\Ric^{\tilde M_k}_z=\Ric_z$ for every
$z\in D_m$ and
\begin{align*}
\ric^{\tilde M_k}_{U_s^k}(h(s))=\ric^{\tilde M_k}_{U_s^k}(h(s))\1_{\{s<\tau_m\}}=
\ric_{U_s}(h(s)),\quad \forall\ k>m.
\end{align*}

Moreover, we observe that for any compact set $K\subset D_m$ and $q>0$,
\begin{equation}\label{eq3.13}
\begin{split}
\sup_{x\in K}\Ee_x\left[
\int_0^{1\wedge \tau_m} \left| \ric_{U_s}(U_0^{-1}v)\right|^q \d s\right]
&\le |v|^q \sup_{x\in K}\Ee_x\left[\int_0^{1\wedge \tau_m}\left|\ric_{U_s}\1_{\{s<\tau_m\}}\right|^q \d s\right]\\
&\le |v|^q \sup_{z\in D_m}\|{\color{blue} \Ric_z}\|^q<\infty.
\end{split}
\end{equation}
Combining this with \eqref{eq2.4} and the fact that $h(s)\neq 0$ only if $s\le t\wedge \tau_m=t\wedge \tau_{m}^k$ yields immediately that
\begin{equation}\label{add4}
\sup_{x\in D_m, v\in T_x M, |v|=1}\Ee_x\left[\int^t_0|\Theta_s^{h}|^2ds \right]<\infty,\quad \ \forall\ t>0.
\end{equation}
Thus, applying  Proposition \ref{prp4.5} to $P_t^kf$ (note that $\tilde M_k$ is compact) and using \eqref{add1} we obtain that
for all $v\in T_x M$,
\begin{equation}\label{add2}
\begin{split}
\langle \nabla P^k_tf(x), v\rangle_{T_x M}&=-\EE_x\left[ f(X^k_t)\int_0^t \<\Theta_s^{h, k}, dB_s\> \right]=-\EE_x\left[ f(X^k_t)\int_0^t \<\Theta_s^{h}, dB_s\>\right].
\end{split}
\end{equation}
For any function $\psi \in C_c^{\infty}(M)$ and vector field  $V \in C_c^{\infty}(M;TM)$ with supports in $ D_m$
satisfying that $|V(x)|\le 1$ for all $x\in D_m$, we can use $\nabla$ and $\d x$ for the  the gradient operator and the Riemannian volume measure on both manifolds $M$ and $\tilde M_k$, so we have
\begin{equation}\label{add3}
\begin{split}
&\int_{M} \EE_x\left[ f(X^k_t)\int_0^t \<\Theta_s^{h(x)}, dB_s\> \right]\psi(x)\d x
\\&=\int_{M} \langle \nabla P^k_tf(x), V(x)\rangle_{T_x M}\psi(x)\d x\\
&=\int_{\tilde M_k} \langle \nabla P^k_tf(x), V(x)\rangle_{T_x M}\psi(x)\d x\\
&=-\int_{\tilde M_k} \Ee_x\left[f\left(X_t^{k}\right)\right]\div(V\psi)(x)\d x=-\int_{M}
\Ee_x\left[f\left(X_t^{k}\right)\right]\div(V\psi)(x)\d x.
\end{split}
\end{equation}
Here $h(x)$ is defined by \eqref{l5-2-1a} with $v=V(x)$.

Meanwhile note that $X_t=X_t^k$ if $t<\tau_k$, for every $x\in D_m$ it holds
\begin{equation*}
\begin{split}
&\lim_{k \rightarrow \infty}
\left|\Ee_x\left[f(X_t^{k})\int_0^t \<\Theta_s^{h(x)}, dB_s\>\right]
-\Ee_x\left[f(X_t)\int_0^t \<\Theta_s^{h(x)}, dB_s\>\1_{\{t<\zeta\}}\right]\right|\\
&\le \lim_{k \rightarrow \infty}
\Ee_x\left[\left|f(X_t^{k})-f(X_t)\1_{\{t<\zeta\}}\right|
\left|\int_0^t \<\Theta_s^{h(x)}, dB_s\>\right|\right]\\
&\le \lim_{k \rightarrow \infty}\sqrt{\Ee_x\left[\left|f(X_t^{k})-f(X_t)\1_{\{t<\zeta\}}\right|^2\right]}
\sqrt{\Ee_x\left[\left|\int_0^t \<\Theta_s^{h(x)}, dB_s\>\right|^2\right]}\\
&\le \lim_{k \rightarrow \infty}\sqrt{2}C\|f\|_{\infty}\sqrt{\Pp_x\left(\tau_k\le t<\zeta\right)}=0,
\end{split}
\end{equation*}
where
$$C:=\sup_{x\in D_m, v\in T_x M, |v|=1}\Ee_x\left[\int^t_0|\Theta_s^{h}|^2ds \right]$$ is finite for every $t>0$,
which is due to \eqref{add4}. With this we may take $k\to \infty$ in \eqref{add3}, then
$$\Ee_x\left[f(X_t^{k})\int_0^t \<\Theta_s^{h(x)}, dB_s\>\right]\to  \Ee_x\left[f(X_t)\int_0^t \<\Theta_s^{h(x)}, dB_s\>\1_{\{t<\zeta\}}\right]$$
and consequently
\begin{equation*}
\begin{split}
& \int_{D_m} \Ee_x\left[f\left(X_t\right)\int_0^t \<\Theta_s^{h(x)}, dB_s\>\1_{\{t<\zeta\}}\right]\psi(x)\d x
=\int_{D_m} \Ee_x\left[f\left(X_t\right)\1_{\{t<\zeta\}}\right]\div(V\psi)(x)\d x.
\end{split}
\end{equation*}
Since $m$ is arbitrary, so it follows that for all test vector fields $V \in C_c^{\infty}(M;TM)$ and test functions $\psi \in C_c^{\infty}(M)$,
\begin{equation*}
\begin{split}
& \int_{M} \Ee_x\left[f\left(X_t\right)\int_0^t \<\Theta_s^{h(x)}, dB_s\>\1_{\{t<\zeta\}}\right]\psi(x)\d x
=\int_{M} P_tf(x)\div(V\psi)(x)\d x,
\end{split}
\end{equation*}
which means that  the weak (distributional) gradient
$\nabla P_t f$ exists
\begin{equation*}
\begin{split}
\left\langle \nabla P_t f(x), V(x)\right\rangle_{T_x M}=\Ee_x\left[f\left(X_t\right)
\int^t_0\<\Theta_s^{h(x)}, dB_s\>\1_{\{t<\zeta\}}\right],\quad \ x\in M.
\end{split}
\end{equation*}
According to the same arguments in the proof of Lemma \ref{lem7.2} in the Appendix, the functional $x \mapsto \Ee_x\Big[f\left(X_t\right)$
$\int^t_0\<\Theta_s^{h(x)}, dB_s\>\1_{\{t<\zeta\}}\Big]$ is continuous.
So we have verified that the distributional derivative $\nabla P_t f$ exists and is continuous, then $\nabla P_t f$ is
the classical gradient and expression \eqref{l5-2-1} holds.
\end{proof}

Now we present an  estimate
for the difference between the gradients of logarithmic heat kernels. Note that for every
$x, y \in K \subset B_o(2m-2)\subset D_m$,
we could view  $\nabla_x\log p(t,x,y)$ and
$\nabla_x\log p_{\tilde M_{m}}(t,x,y)$ as vectors in $T_x M$, so that
$$|\nabla_x\log p(t,x,y)-\nabla_x \log p_{\tilde M_{m}}(t,x,y)|_{T_x M}$$
is well defined.

Let $\partial$ be the cemetery point. We make the convention that $p(t,\partial,y)=0$ for all $t$.
\begin{prp}\label{prp6.6}
Suppose that  $K$ is a compact subset of $M$ and $L>1$ is a positive number.
Then for every suficiently large $m$,  we could find a number $t_0(K,L,m)$, depending on $K,L, m$, such that for every $0<t\le t_0$,
\begin{equation}\label{eq3.7}
\sup_{x,y \in K} \left|\nabla_x\log p(t,x,y)- \nabla_x \log p_{\tilde M_{m}}(t,x,y)\right|_{T_x M}
\le C(m)e^{-\frac{L}{t}},
\end{equation}
where $C(m)$ is a positive constant, which may depend on $m$.
\end{prp}
\begin{proof}
Let us  fix  points $x,y \in K$ and a unit vector $v \in T_x M$.
Let $m $ be a natural number such that  $B_o(2m-2)\supset K$. Let $t>0$ be fixed

By \eqref{l5-2-1} where  $\Theta^h=h'(t)+\f 12 \ric_{U_t}(h(t))$, we have, for every $f \in C_c(M)$,
\begin{equation*}
\begin{split}
\langle \nabla P_t f(x),v\rangle_{T_x M}
=&\frac{2}{t}\Ee_x\Bigl[\int_0^{\frac{t}{2}} \left\langle l_m(s)U_0^{-1}v,\d B_s
\right\rangle \; f\bigl(X_t\bigr)\1_{\{t<\zeta\}}\Bigr]\\
&-\Ee_x\Bigl[\int_0^{\frac{t}{2}} \Bigl\langle
\bigl(\frac{t-2s}{t}\bigr)l_m'(s)U_0^{-1}v, \d B_s \Bigr\rangle\;f\left(X_t\right)\1_{\{t<\zeta\}}\Bigr]\\
&-\frac{1}{2}\Ee_x\Bigl[\int_0^{\frac{t}{2}}\Bigl\langle \, \ric_{U_t}\Bigl(\bigl(\frac{t-2s}{t}\bigr)l_m(s)U_0^{-1}v\Bigr), \d B_s\Bigr\rangle
\;f\left(X_t\right)\1_{\{t<\zeta\}}\Bigr]
\end{split}
\end{equation*}
Since $f$ has compact support, the indicator function $\1_{\{t<\zeta\}}$ can be removed.
Taking the conditional expectation on $\sigma(X_t)$,  we obtain, for all $x \in M$ and
almost everywhere $y \in M$ (with respect to volume measure on $M$),

 \begin{equation}
\begin{split}
&t\bigl\langle \nabla_x \log p(t,x,y),v\bigr\rangle_{T_x M}\\
& =t\frac{\langle  \nabla_x p(t,x,y), v \rangle_{T_x M}}{p(t,x,y)}
=2\Ee_x\Bigl[\1_{\{t<\zeta\}}\int_0^{\frac{t}{2}} \left\langle l_m(s)U_0^{-1}v, \d B_s
\right\rangle\; |\; X_t=y\Bigr]\\
&\qquad \qquad \qquad -t\Ee_x\left[\1_{\{t<\zeta\}}\int_0^{\frac{t}{2}} \Bigl\langle
\bigl(\frac{t-2s}{t}\bigr)l_m'(s)U_0^{-1}v, \d B_s\Bigr\rangle  \; \Big|X_t=y\right]\\
&\qquad \qquad \qquad -\frac{t}{2}\Ee_x\left[\1_{\{t<\zeta\}}\int_0^{\frac{t}{2}}\Bigl\langle \, \ric_{U_t}\bigl(\frac{t-2s}{t}\bigr)l_m(s)
U_0^{-1}v, \d B_s\Bigr\rangle
\;\Big|\;X_t=y\right]\\
&=\Ee_x\Bigl[\1_{\{t<\zeta\}}\int_0^{\frac{t}{2}} g_m(s)\bigl\langle
U_0^{-1}v, dB_s
\bigr\rangle\; |\; X_t=y\Bigr].\end{split}
\end{equation}
where  $$g_m(s):=2l_m(s) -t
\left(\frac{t-2s}{t}\right)l_m'(s)
 -\frac{t}{2}\;\ric_{U_t}\left(\frac{t-2s}{t}\right)l_m(s).$$
 Thus,
  \begin{equation}\label{e3-8}
\begin{split}
&t\bigl\langle \nabla_x \log p(t,x,y),v\bigr\rangle_{T_x M}
=\Ee_x\Bigl[\1_{\{t<\zeta\}}\int_0^{\frac{t}{2}} g_m(s)\bigl\langle
U_0^{-1}v, dB_s
\bigr\rangle\; |\; X_t=y\Bigr]\\
&=\Ee_x\left[\int_0^{\frac{t}{2}} g_m(s)\bigl\langle
U_0^{-1}v, dB_s
\bigr\rangle\;\frac{p\left(\frac{t}{2},X_{\frac{t}{2}},y\right)}{p(t,x,y)}\1_{\{\f t 2<\zeta\}} \right]\\
&=\Ee_x\left[\int_0^{\frac{t}{2}} g_m(s)\bigl\langle
U_0^{-1}v, dB_s
\bigr\rangle\;\frac{p\left(\frac{t}{2},X_{\frac{t}{2}},y\right)}{p(t,x,y)} \right].
\end{split}
\end{equation}
We have used the property that for $p(\frac{t}{2},X_{\frac{t}{2}},y)=0$  whenever $\f t 2\ge\zeta(x)$.

Based on the  heat kernel estimates in the previous lemmas, by the proof of Lemma \ref{lem7.2} we know immediately
$$x\mapsto\Ee_x\left[\int_0^{\frac{t}{2}} g_m(s)\bigl\langle
U_0^{-1}v, dB_s
\bigr\rangle\;\frac{p\left(\frac{t}{2},X_{\frac{t}{2}},y\right)}{p(t,x,y)}\right]$$
is  continuous.
So the expression
above is true for all $x,y \in M$.

Since  $l_m'\left(s,X_{\cdot}^m\right)=
l_m'\left(s, X_{\cdot}\right)$ and $l_m\left(s,X_{\cdot}^m\right)=
l_m\left(s, X_{\cdot}\right)$ for $s<\tau_m$ and $X_s=X_s^m$,
applying the same arguments above to $\tilde M_{m}$ we have
\begin{equation}\label{e3-9}
\begin{split}
&\langle t\nabla \log p_{\tilde M_{m}}(t,x,y),v \rangle_{T_x M}
=\Ee_x\left[\int_0^{\frac{t}{2}} g_m(s)\bigl\langle
U_0^{-1}v, dB_s
\bigr\rangle\;\frac{p_{\tilde M_m}\left(\frac{t}{2},X_{\frac{t}{2}}^m,y\right)}{p_{\tilde M_m}(t,x,y)}\right].
\end{split}
\end{equation}

To apply Lemma \ref{lem6.4} it remains to make moment estimates for  $\int_0^t g(s)\<v , U_0dB_s\>$.
For any $m\in \N$ large enough and $q>0$, \eqref{eq3.13} implies that condition \eqref{l4-1-0} in Lemma \ref{lem6.4} holds for the process $\Upsilon_t=\ric_{U_t}$ and we could apply \eqref{r4-1-1} and \eqref{l4-1-1} to conclude the estimates.
\end{proof}

We are now in a position to proceed to prove the gradient estimates for $\log p(t,x,y)$.

\begin{thm}\label{thm6.7}
The following statements hold.
\begin{enumerate}
\item [(1)] Suppose $x,y \in M$ and $x\notin \text{Cut}_M(y)$, then
\begin{equation}\label{eq3.14}
\lim_{t \downarrow 0}t\nabla_x\log p(t,x,y)=-\nabla_x\left(\frac{d^2(x,y)}{2}\right).
\end{equation}
Here the convergence is uniformly in $x $ on
any compact subset of $M\setminus \text{Cut}_M(y)$.

\item[(2)] Let $K$ be a compact subset of $M$. Then there exists a
positive constant $C(K)$, which may depend on $K$, such that
\begin{equation}\label{eq3.15}
\left|\nabla_x \log p(t,x,y)\right|_{T_x M}\le C(K)\left(\frac{d(x,y)}{t}+\frac{1}{\sqrt{t}}\right),\quad x,y\in K,
\ t \in (0,1].
\end{equation}
\end{enumerate}
\end{thm}

\begin{proof}
In the proof, the constant $C$ (which depends on $\tilde K$ or $K$) may change from line to line.
For every $m \in \N $ with $K\subset B_o(2m-2)\subset D_{m}$, we have
\begin{equation}\label{eq3.16}
\begin{split}
&t\nabla_x \log p(t,x,y)\\&=
t\nabla_x \log p_{\tilde M_m}(t,x,y)+\left(t\nabla_x \log p(t,x,y)-t\nabla_x \log p_{\tilde M_m}(t,x,y)\right).
\end{split}
\end{equation}
For each compact set $\tilde K \subset  M\setminus \text{Cut}_M(y)$, by \eqref{eq3.7}\emph{} we could
choose a $m_0 \in \N $ large enough such that $\tilde K \subset B_o(2m_0-2)\subset D_{m_0}$ and
\begin{equation*}
\lim_{t \downarrow 0}\sup_{x \in \tilde K}
|t\nabla_x \log p(t,x,y)-t\nabla_x \log p_{\tilde M_{m_0}}(t,x,y)|_{T_x M}=0.
\end{equation*}
At the same time, since $\tilde M_{m_0}$ is compact and
$\tilde K$ is outside of the cut locus $\text{Cut}_{\tilde M_m}(y)$, we have
\begin{equation*}
\begin{split}
&\lim_{t \downarrow 0}\sup_{x \in \tilde K}
\left|t\nabla_x \log p_{\tilde M_{m_0}}(t,x,y)
+\nabla_x \left(\frac{d^2(x,y)}{2}\right)\right|_{T_x M}\\
&=
\lim_{t \downarrow 0}\sup_{x \in \tilde K}
\left|t\nabla_x \log p_{\tilde M_{m_0}}(t,x,y)
+\nabla_x \left(\frac{d_{\tilde M_{m_0}}^2(x,y)}{2}\right)\right|_{T_x \tilde M_{m_0}}=
0.
\end{split}
\end{equation*}
In the first step, we used that $d_{\tilde M_{m_0}}(x,y)=d(x,y)$ for $x,y \in \tilde K$,
 while the second step is due to Corollary 2.29 from Malliavin and Stroock \cite{MS},
 (see also Bismut \cite{Bis} and Norris \cite{Norris}). Plugging this into \eqref{eq3.16} with $m=m_0$, then we have shown \eqref{eq3.14}.

Given a compact set $K \subset M$ and a constant $L>1$, based on \eqref{eq3.7} there exists
a sufficiently large natural number $m_0$  such that $K \subset B_o(2m_0-2)\subset D_{m_0}$ and
and $t_0\in (0,1)$ such that
\begin{equation}\label{t3-5-1}
\sup_{x,y \in K}|I^{m_0}(t,x,y)|_{T_x M}\le Ce^{-\frac{L}{t}},\ \quad \forall \; t\in (0,t_0].
\end{equation}
Since $\tilde M_{m_0}$ is compact, 
we can apply Hsu \cite[Theorem 5.5.3]{Hsu2} or  Sheu \cite{Sh}
to show that for all $x,y\in K$ and $t\in (0,1]$,
\begin{equation}\label{t3-5-2}
\begin{split}
\left|\nabla_x \log p_{\tilde M_{m_0}}\left(t,x,y\right)\right|_{T_x M}&\le
C(K)\left(\frac{d_{\tilde M_{m_0}}(x,y)}{t}+\frac{1}{\sqrt{t}}\right)\\
&=C(K)\left(\frac{d(x,y)}{t}+\frac{1}{\sqrt{t}}\right).
\end{split}
\end{equation}
Combing \eqref{t3-5-1} and \eqref{t3-5-2} into \eqref{eq3.16} with $m=m_0$ we immediately find \eqref{eq3.15}
holds for all $t \in (0,t_0]$.

Also note that for all $ x,y\in K$ and for all $\ t\in [t_0,1]$,
\begin{equation*}
\begin{split}
\left|\nabla_x \log p\left(t,x,y\right)\right|_{T_x M}\le C(K,t_0)
\le C\left(\frac{d(x,y)}{t}+\frac{1}{\sqrt{t}}\right).
\end{split}
\end{equation*}
By now we have completed the proof of \eqref{eq3.15}.
\end{proof}

{\bf Remark:}
\begin{itemize}
\item [(1)] The gradient estimate \eqref{eq3.15} was proved in \cite{S,Sh,Hsu2}  for a complete manifold  with Ricci curvature bounded from below by a constant $C_0$.
In that case, the constant $C(K)$ in \eqref{eq3.15} in uniform and only depends on $C_0$,  see also \cite{XLi}
for the case of the estimates for heat kernel associated with the Witten Laplacian operator.

\item [(2)] By carefully tracking the proof, we know the constant $C(K)$ from \eqref{eq3.15} depends only on $C_1(m_0)$, $\inf_{x\in D_{m_0}}\|{\rm Ric}_x\|$, and
$\sup_{x\in D_{m_0}}\Ee_x \int_0^{1}\left|l'_{m_0}(s)\right|^2\d s$, where $C_1(m_0)$ is the positive constant such that
\begin{equation*}
\begin{split}
|\nabla_x \log p_{\tilde M_{m_0}}(t,x,y)|_{T_x M}&\le C_1(m_0)\left(\frac{d_{\tilde M_{m_0}}(x,y)}{t}+\frac{1}{\sqrt{t}}\right)\\
&=C_1(m_0)\left(\frac{d(x,y)}{t}+\frac{1}{\sqrt{t}}\right).
\end{split}
\end{equation*}

\end{itemize}

\subsection{Proof of Theorem \ref{thm3.1} and the main theorem:  Hessian estimates}
\label{proof-second-order}
Now we can prove the claim for the second order gradient of logarithmic heat kernel.
In Proposition \ref{prp4.5}, we have established a second order gradient formula for $P_t f$ on a compact manifold.
In its proof we exchanged the differential and the integral operators several times, which may not
hold if $M$ is not compact. So it is not trivial to extend  Proposition \ref{prp4.5} to a non-compact manifold.

To prove Theorem \ref{thm3.1}, we begin with comparing the terms in $\nabla^2 P_t^k$ and $\nabla^2P_t$.

\begin{lem}\label{lemma-I}
Given a point $x\in M$ and a vector $v\in T_x M$, suppose that $m$ is sufficiently large so $x\in D_m$ and $k>m$. Let  $\{U_t^k\}_{t\ge 0}$ be the horizontal Brownian motion on $\tilde M_k$  as defined in (\ref{e4-1}).  Set $X_t^k=\pi(U_t^k)$ and $P_t^k f(x)=\Ee_x\left[f(X_t^k)\right]$.
Let $h(s)=\left(\frac{t-2s}{t}\right)^+ \cdot l_m(s,X_\cdot)
\cdot U_0^{-1}v$ and define
\begin{equation}\label{p4-1-2}
I \left(\frac{t}{2},X_{\cdot},v\right)
:=\left(\int_0^{\frac{t}{2}}\langle \Theta^h_s, \d B_s\rangle\right)^2-
\int_0^{\frac{t}{2}} \langle \Lambda^h_s, \d B_s\rangle-\int_0^{\frac{t}{2}} \left|\Theta_s^h\right|^2 \d s.
\end{equation}

Let $I\bigl(\frac{t}{2},X_{\cdot}^k,v\bigr)$ be defined with the corresponding terms in $\tilde M_k$.
Then we have
\begin{equation}\label{p4-1-2a}
I\left(\frac{t}{2},X_{\cdot}^k,v\right)=I\left(\frac{t}{2},X_{\cdot},v\right)
=\left(\int_0^t\langle \Theta_s^h, \d B_s\rangle\right)^2-\int_0^t \langle \Lambda_s^h, \d B_s\rangle-\int_0^t \left|\Theta_s^h\right|^2 \d s.
\end{equation}
Furthermore it holds that
\begin{equation}\label{p4-1-4}
\sup_{x\in D_m, v\in T_x M, |v|=1}\Ee_x\left[\left|I\left(\frac{t}{2},X_{\cdot},v\right)\right|^2\right]<\infty,\quad \ \forall\ t>0.
\end{equation}
\end{lem}

\begin{proof}
Let $\Theta_s^{h,k}$, $\Gamma_s^{h,k}$, $\Lambda_s^{h,k}$ be the corresponding terms of $\Theta_s^h$,
$\Gamma_s^h$, $\Lambda_s^h$ defined on
$\tilde M_k$. By \eqref{add1} we have
\begin{equation}\label{p4-1-3}
\Theta_s^{h,k}=h'(s)+\ric^{\tilde M_k}_{U_s^k}(h(s))=h'(s)+\ric_{U_s}(h(s))=\Theta_s^h,\quad \forall\ k>m.
\end{equation}

Still based on \eqref{e3-2}, \eqref{e3-2a} and the same arguments for \eqref{p4-1-3} we can obtain that
\begin{equation*}
\Gamma_s^{h,k}=\Gamma_s^h,\quad \Lambda_s^{h,k}=\Lambda_s^h,\qquad \forall\ k>m.
\end{equation*}
Therefore the term $I\left(\frac{t}{2},X_{\cdot}^k,v\right)$ in \eqref{p4-1-2} is independent of $k$
and the required identity \eqref{p4-1-2a} holds.
Finally, (\ref{p4-1-4}) immediately follows from  the moment estimates \eqref{eq2.4} for $l_m'$
 and the same arguments for \eqref{eq3.13}.
\end{proof}

We can now begin the
\ \newline\emph{Proof of Theorem \ref{thm3.1}.} The idea of the proof is similar to that of Lemma \ref{lem6.5}.
For convenience of the reader, here we provide a detailed proof.
Let $m_0\in \N$ satisfy that $x\in D_{m_0+1}$, then for every $k>m>m_0$ it holds that
 $B_o(2m-2)\subset D_m \subset D_k$. Let $h(s)=\left(\frac{t-2s}{t}\right)^+ \cdot l_m(s,X_\cdot)\cdot U_0^{-1}v=\left(\frac{t-2s}{t}\right)^+ \cdot l_m(s,X_\cdot^k)\cdot U_0^{-1}v$.
We can apply \eqref{p3-1-1} in Proposition \ref{prp4.5}
to the compact manifold $\tilde M_k$ to obtain that for every $k>m$,
\begin{equation}\label{p4-1-2-d}
\begin{split}
&\quad \left\langle \nabla^2 P_t^k f(x), v\otimes v\right\rangle_{T_x M \otimes T_x M}\\
&=\EE_x\left[f(X_t^k)\left(
\left(\int_0^t\langle \Theta_s^{h,k}, \d B_s\rangle\right)^2-
\int_0^t \langle \Lambda^{h,k}_s, \d B_s\rangle-\int_0^t \left|\Theta^{h,k}_s\right|^2 \d s\right)\right]\\
&=\Ee_x\left[f(X_t^k)I\left(\frac{t}{2},X_{\cdot}^k,v\right)\right]=
\Ee_x\left[f(X_t^k)I\left(\frac{t}{2},X_{\cdot},v\right)\right],
\end{split}
\end{equation}
where the process $\Theta_s^{h,k}$, $\Lambda_s^{h,k}$ are defined by \eqref{e3-2}, \eqref{e3-2a} on $\tilde M_k$,
and in the last step
we have applied \eqref{p4-1-2a}.

According to \eqref{p4-1-2} and integration by parts formula (on compact manifold $\tilde M_k$), for any $\psi \in C_c^{\infty}(M)$, $V \in C_c^{\infty}(M;TM)$ with
$\supp{\psi}\cup\supp{V}\subset D_m$ we have
\begin{equation}\label{eq3.26}
\begin{split}
&\int_{M} \Ee_x\left[f\left(X_t^{k}\right)I\left(\frac{t}{2},X_{\cdot},V(x)\right)\right]\psi(x)\d x
\\&=\int_{M} \left\langle \nabla^2 P_t^{k} f(x), V(x)\otimes V(x)\right\rangle_{T_x M \otimes T_x M}\psi(x)\d x\\
&=\int_{\tilde M_k} \left\langle \nabla^2 P_t^{k} f(x), V(x)\otimes V(x)\right\rangle_{T_x M \otimes T_x M}\psi(x)\d x\\
&=\int_{\tilde M_k} \Ee_x\left[f\left(X_t^{k}\right)\right]\Psi(\psi,V)(x)\d x=\int_{M}
\Ee_x\left[f\left(X_t^{k}\right)\right]\Psi(\psi,V)(x)\d x.
\end{split}
\end{equation}
Here we denote the gradient operator and Riemannian volume measure on both $M$ and $\tilde M_k$ by
$\nabla$ and $\d x$, and we set
\begin{equation*}
\begin{split}
\Psi(\psi,V)(x):=&\div(\div(V\psi) V)(x)+\div(\psi\nabla_V V)(x)\\
=&\psi(x)\left(\div\left(\nabla_V V\right)+\left(\div V\right)^2+
\left\langle V, \nabla \div V\right\rangle_{T_x M}\right)(x)\\
&+2\left\langle \nabla \psi, \nabla_V V+(\div V) V\right\rangle_{T_x M}(x)+
\left\langle \nabla^2 \psi(x), V(x)\otimes V(x)\right\rangle_{T_xM \otimes T_x M}.\\
\end{split}
\end{equation*}
The second and last step above follow from the properties that Riemannian volume measure $\d x$ and
the second order gradient operator
$\nabla^2$ on $M$ are the same as that on $\tilde M_k$,
when they are restricted on $D_m$, the third equality is due to the integration
by parts formula.
Meanwhile note that $X_t=X_t^k$ if $t<\tau_k$, for every $x\in D_m$ it holds
\begin{equation*}
\begin{split}
&\lim_{k \rightarrow \infty}
\left|\Ee_x\left[f(X_t^{k})I\left(\frac{t}{2},X_{\cdot},V(x)\right)\right]
-\Ee_x\left[f(X_t)I\left(\frac{t}{2},X_{\cdot},V(x)\right)\1_{\{t<\zeta\}}\right]\right|\\
&\le \lim_{k \rightarrow \infty}
\Ee_x\left[\left|f(X_t^{k})-f(X_t)\1_{\{t<\zeta\}}\right|
\left|I\left(\f t 2,X_{\cdot},V(x)\right)\right|\right]\\
&\le \lim_{k \rightarrow \infty}\sqrt{\Ee_x\left[\left|f(X_t^{k})-f(X_t)\1_{\{t<\zeta\}}\right|^2\right]}
\sqrt{\Ee_x\left[\left|I\left(\f t 2,X_{\cdot},V(x)\right)\right|^2\right]}\\
&\le \lim_{k \rightarrow \infty}\sqrt{2}C\|f\|_{\infty}\sqrt{\Pp_x\left(\tau_k\le t<\zeta\right)}=0,
\end{split}
\end{equation*}
where the last inequality is due to \eqref{p4-1-4}.

Putting this into \eqref{eq3.26}, letting $k \rightarrow \infty$ we see that
for every $\psi \in C_c^{\infty}(M)$ and $V \in C^{\infty}(M;TM)$ with
$\supp{\psi}\cup\supp{V}\subset D_m$,
\begin{equation*}
\begin{split}
& \int_{D_m} \Ee_x\left[f\left(X_t\right)I\left(\frac{t}{2},X_{\cdot},V(x)\right)\1_{\{t<\zeta\}}\right]\psi(x)\d x
=\int_{D_m} \Ee_x\left[f\left(X_t\right)\1_{\{t<\zeta\}}\right]\Psi(\psi,V)(x)\d x,
\end{split}
\end{equation*}
which implies the weak (distributional) second order gradient
$\nabla^2 P_t f$ exists on $D_m$ and
\begin{equation}\label{eq3.27}
\begin{split}
\left\langle \nabla^2 P_t f(x), v\otimes v\right\rangle_{T_x M \otimes T_x M}=\Ee_x\left[f\left(X_t\right)
I\left(\frac{t}{2},X_{\cdot},v\right)\1_{\{t<\zeta\}}\right],\ \ x\in D_m, v \in T_x M.
\end{split}
\end{equation}
As shown by Lemma \ref{lem7.2} in Appendix, the functional $x \mapsto \Ee_x\left[f\left(X_t\right)
I\left(\frac{t}{2},X_{\cdot},V(x)\right)\1_{\{t<\zeta\}}\right]$ is continuous.
Now the distributional derivative $\nabla^2 P_t f$ exists and is continuous, then $\nabla^2 P_t f$ is
the classical second order gradient on $D_m$ and expression \eqref{t3-1-1} holds.

\begin{prp}\label{prp6.9}
Suppose that  $K$ is a compact subset of $M$ and $L>1$ is a positive constant.
Then, for any sufficiently large $m$, we could find a
$t_0(K,L,m)$ such that for any $t\in (0,t_0]$,
\begin{equation}\label{eq3.17}
\sup_{x,y\in K}e^{\frac{L}{t}}\left|t\nabla_x^2\log p(t,x,y)-t \nabla_x^2 \log p_{\tilde M_{m}}(t,x,y)
\right|_{T_x M\otimes T_x M}\le C(m)e^{-\frac{L}{t}}
\end{equation}
where $C(m)$ is a positive constant which may depend on $m$.
\end{prp}
\begin{proof}

Let us fix $x,y \in K$ and a unit vector $v \in T_x M$ with $|v|=1$. Suppose that $m \in \N $ such that $K \subset B_o(2m-2)\subset D_m$.
Then by \eqref{t3-1-1} we have
\begin{equation*}
\left\langle \nabla^2 P_t f(x), v\otimes v\right\rangle_{T_x M \otimes T_x M}=\Ee_x\left[f\left(X_t\right)
I\left(\frac{t}{2},X_{\cdot},v\right)\1_{\{t<\zeta\}}\right],
\end{equation*}
where $I\left(\frac{t}{2},X_{\cdot},v\right)$ is defined by \eqref{p4-1-2} with
$h(s):=\Big(\frac{t-2s}{t}\Big)^+\cdot l_m\left(s,X_{\cdot}\right)\cdot U_0^{-1}v$.

By this representation and following the same arguments of \eqref{e3-8} and \eqref{e3-9} we obtain
\begin{equation*}
\begin{split}
& \frac{\left\langle \nabla^2_x p(t,x,y), v\otimes v\right\rangle_{T_x M \otimes T_x M}}{p(t,x,y)}=
\Ee_x\left[I\left(\frac{t}{2},X_{\cdot},v\right)
\frac{p\left(\frac{t}{2},X_{\frac{t}{2}},y\right)}{p(t,x,y)} \right],\\
& \frac{\left\langle \nabla^2_x p_{\tilde M_m}(t,x,y), v\otimes v\right\rangle_{T_x M\otimes T_x M}}{p_{\tilde M_m}(t,x,y)}=
\Ee_x\left[I\left(\frac{t}{2},X_{\cdot},v\right)
\frac{p_{\tilde M_m}\left(\frac{t}{2},X_{\frac{t}{2}}^m,y\right)}{p_{\tilde M_m}(t,x,y)}\right].
\end{split}
\end{equation*}
Based on above expression and following the same arguments in the proof of Proposition \ref{prp6.6} (especially applying \eqref{l4-1-1} and
\eqref{r4-1-1}--\eqref{r4-1-3}) we could find a $m_0(K,L)\in \N$ such that
for all $m\ge m_0$, there exists a $t_0(K,L,m)>0$ such that
\begin{equation}\label{p4-3-1}
\sup_{x,y\in K}\left|\frac{\nabla^2_x p(t,x,y)}{p(t,x,y)}-\frac{\nabla^2_x p_{\tilde M_m}(t,x,y)}{p_{\tilde M_m}(t,x,y)}\right|_{T_x M \otimes T_x M}
\le C(m)e^{-\frac{L}{t}},\quad \ t\in (0,t_0].
\end{equation}

Meanwhile we have
\begin{equation*}
\begin{split}
&\left\langle\nabla^2_x \log p(t,x,y),v\otimes v\right\rangle_{T_x M \otimes T_x M}
=\frac{\left\langle \nabla_x^2 p(t,x,y),  v\otimes v\right\rangle_{T_x M \otimes T_x M}}{p(t,x,y)}
-\left(\frac{\langle \nabla_x p(t,x,y),v\rangle_{T_x M}}{p(t,x,y)}\right)^2,\\
\end{split}
\end{equation*}
and the similar expression holds for  $\left\langle\nabla^2_x \log p_{\tilde M_m}(t,x,y), v\otimes v\right\rangle_{T_x M \otimes T_x M}$. Together with \eqref{eq3.7} and \eqref{p4-3-1}, this yields \eqref{eq3.17} and concludes the proof.
\end{proof}

With \eqref{eq3.17} we are in the position to prove the second part of the main theorem on the short time and asymptotic second order gradient estimates.

\begin{thm}\label{thm6.10}
The following statements hold.
\begin{enumerate}
\item [(1)] Suppose $y \in M$ and $\tilde K \subset M\setminus \text{Cut}_M(y)$ is a compact set, then
\begin{equation}\label{eq3.28}
\lim_{t \downarrow 0}\sup_{x \in \tilde K}
\left|t\nabla_x^2\log p(t,x,y)+\nabla_x^2\left(\frac{d^2(x,y)}{2}\right)
\right|_{T_x M \otimes T_x M}=0.
\end{equation}

\item[(2)]  For each $y\in M$ and $\delta<i(y)$ there exist positive constants
$t_0$ and $C_1$ such that
\begin{equation}\label{eq3.29}\aligned
&\left|t\nabla_x^2 \log p(t,x,y)+\textbf{I}_{T_x M}\right|_{T_x M \otimes T_x M}\\
&\le C_1\left(d(x,y)+
\sqrt{t}\right),\quad \quad \ x\in B_y(\delta),\ t \in (0,t_0],\endaligned
\end{equation}
where $\textbf{I}_{T_x M}$ is the identical map on $T_x M$.

\item[(3)] Suppose $K \subset M$ is a compact subset of $M$, then there exists a
positive constant $C_2(K)$, such that
\begin{equation}\label{eq3.30}
\left|\nabla_x^2 \log p(t,x,y)\right|_{T_x M \otimes T_x M}\le C_2\left(\frac{d^2(x,y)}{t^2}+\frac{1}{t}\right),\quad  x,y\in K,
\ t \in (0,1].
\end{equation}
\end{enumerate}
\end{thm}

\begin{proof} 
By Malliavin and Stroock \cite[Corollary 2.29]{MS},  Gong and Ma \cite[Theorem 3.1]{GM} and Stroock \cite{S} (or Sheu \cite{Sh}), we know
\eqref{eq3.28}--\eqref{eq3.30}  hold when $M$ is compact.
Then using the estimates \eqref{eq3.17} and following the same procedure  as in the proof of
Theorem \ref{thm6.7} we can verify that  \eqref{eq3.28}-\eqref{eq3.30} hold for any complete
Riemannian manifold.
\end{proof}

\begin{appendix}
\section*{Approximation procedure}\label{appn}
Let $(M,g)$ and $D_m\subset M$ be the same terms in Section \ref{cut-off}.

\begin{lem}\label{lem7.1}
For every $m \in \mathbb{Z}_+$, there exists a (smooth) compact Riemannian manifold $(\tilde M_m,\tilde g_m)$, such that
$(D_m,g)$ is
isometrically embedded into $(\tilde M_m,\tilde g_m)$ as an open set.
In particular, if $y,x \in D_m$ and $x\not \in \cut_y(M)$,  then $x\not \in \cut_y(\tilde M_m)$.
\end{lem}
\begin{proof}
Let $G_m =D_{m+1}$,  recall that $\partial G_m$ is a connected smooth $n-1$-dimensional sub-manifold of $M$.
Hence $\overline{G_m}$ is an $n$-dimensional manifold with smooth boundary,
 then there exist an atlas
of local charts $\{(V_i,\psi_i)\}_{i=1}^N$ of $\overline{G_m}$ such that
\begin{enumerate}
\item [(1)] $\bigcup_{i=1}^N V_i=\overline{G_m}$ ;

\item [(2)] For $i=1, \dots, N_1\le N$, these are charts for the interior. So $V_i\cap \partial G_m=\emptyset$ and
$\psi_i: V_i \rightarrow \mathbf{B}^n:=\{z \in \R^n; |z|< 1\}$ is a smooth diffeomorphsim for all $1\le i\le N_1$;

\item [(3)] For all $i>N_1$,
$V_i\cap \partial G_m\neq\emptyset$, $$\psi_i: V_i \rightarrow
\mathbf{B}^{n,+}:=\{z=(z_1,\dots,z_n) \in \R^n; |z|< 1,z_1\ge 0\}$$ is a smooth diffeomorphsim
and $\psi_i\big(V_i\cap \partial G_m\big)=\partial \mathbf{B}^{n,+}$.
\end{enumerate}
By the Whitney embedding theorem, we could embed $M$ into a (ambient) Euclidean space
 $\R^p$. Let $\hat G_m$ be an identical copy of $G_m$ in $\R^p$ endowed with the local charts
 $\{(\hat V_i,\hat \psi_i)\}_{i=1}^N$ (which is also an identical copy of
$\{(V_i,\psi_i)\}_{i=1}^N$). We define $h:\partial G_m \rightarrow \partial \hat G_m$
by $h(x):=\hat \psi_i^{-1}\left(\psi_i(x)\right)$, if $x \in V_i\cap \partial G_m$,
 $h$ is well defined and is a smooth diffeomorphism.

We glue the boundary of $G_m$ and $\tilde G_m$ together to get
$\tilde M_m:=(G_m\sqcup\hat G_m)/\thicksim$, where $\thicksim$ is an equivalent relation such that
$x \thicksim y$ if and only if $h(x)=y$, $x \in \partial G_m$, $y \in \partial \hat G_m$.
Then $\tilde M_m$ is a smooth compact manifold without boundary. In fact,
$\{(U_i,\phi_i)\}_{i=1}^{N+N_1}=\{(V_i,\psi_i)\}_{i=1}^{N_1}\bigcup \{(\hat V_i,\hat \psi_i)\}_{i=1}^{N_1}$ $\bigcup \{(\tilde V_i,\tilde \psi_i)\}_{i=N_1+1}^N$
is a local charts of $\tilde M_m$. Here $\tilde V_i=(V_i\sqcup\hat V_i)/\thicksim$ for every
$N_1<i\le N$ and
\begin{equation*}
\tilde \psi_i(x)=
\begin{cases}
&\psi_i(x),\ \ \ \ \ \text{if}\ \ x\in V_i,\\
&\mathbf{S}\big(\hat \psi_i(x)\big), \ \ \ \text{if}\ \ x\in \hat V_i,
\end{cases}
\end{equation*}
where $\mathbf{S}:\R^n\rightarrow \R^n$ is a map such that
$\mathbf{S}x=(-x_1,x_2,\dots,x_n)$ for all $x=(x_1,x_2,\dots,x_n)\in \R^n$.
It is easy to see $\tilde \psi_i: \tilde V_i \rightarrow \mathbf{B}^n$, $N_1<i\le N$
is a smooth diffeomorphsim, and the transition map between
different local charts on  $\{(U_i,\phi_i)\}_{i=1}^{N+N_1}$ is smooth.

We construct a smooth Riemannian metric $\tilde g_m$ on $\tilde M_m$
to ensure that $\tilde g_m(z)=g(z)$ for every $z \in D_m$.  For the open set $D_m\subset G_m\subset \tilde M_m$,
by the standard procedure (via the finite local charts) we could construct a
function $\chi_m:\tilde M_m \rightarrow [0,1]$ such that
$\supp\chi_m \subset G_m$ and $\chi_m(x)=1$ for every $x \in D_m$. Note that $G_m$ could also be viewed as
an open subset of $\tilde M_m$, so $\hat g_m(x):=g(x)\chi_m(x)$, $x \in \tilde M_m$ is well defined on $\tilde M_m$.
Fixing a smooth Riemannian metric $g_m^0$ on $\tilde M_m$, which exists, we set
\begin{equation*}
\tilde g_m(x):=g(x)\chi_m(x)+g_m^0(x)\big(1-\chi_m(x)\big),\quad x\in \tilde M_m.
\end{equation*}
It is easy to see $\tilde g_m$ is a smooth Riemannian metric on $\tilde M_m$ and
$\tilde g_m(x)=g(x)$ for each $x \in D_m$. By now we have completed the proof.
\end{proof}

Let $I(t,X_{\cdot},v)$ be as defined in \eqref{p4-1-2}.

\begin{lem}\label{lem7.2}
For every fixed $f \in C_c^{\infty}(M)$, $V \in C^{\infty}(M;TM)$ with compact supports and
$t>0$, the function
\begin{equation*}
F(x):=\Ee\left[f\left(X_t^x\right)I\left(\frac{t}{2},X_{\cdot}^{x},V(x)\right)\1_{\{t<\zeta(x)\}}\right],\quad x\in M.
\end{equation*}
is continuous.
\end{lem}
\begin{proof}
Let $\zeta(x)$ denotes the explosion time of the  solution $X_t^x$ to \eqref{sde1} with the initial value $x$. Let $U$ be a frame at $x$. Then the explosion time of the horizontal Brownian motion
agree with $\xi(x)$ almost surely. So we use $\xi$ for the explosion time of both.
Furthermore, by Elworthy \cite{Elworthy}, there exist a maximal solution  flow $\{U_t(\cdot,\omega)\}_{0\le t <\zeta(\cdot,\omega)}$
to \eqref{sde1} 
such that $U_t(u,\omega)$ is the solution of
 \eqref{sde1} with initial value $u\in OM$, and there is a
null set  $\Delta$  such that for all $\omega \notin \Delta$,

\begin{itemize}
\item [(1)] For each $t>0$, set $\Xi_t(\omega):=\{u\in OM:\ t<\zeta(u,\omega)\}$, Then $\Xi_t$  is open in $OM$
(i.e. $\zeta(\cdot,\omega):OM \to \R_+$ is lower semi-continuous) and $U_t(.,\omega): \Xi_t(\omega)\rightarrow OM$ is a
$C^{\infty}$ diffeomorphism onto its image.

\item[(2)] For each fixed $u \in OM$ with $\pi(u)=x$, there exists a null set $\Delta(u)$ depending on $u$, such that
$\zeta(u,\omega)=\zeta(X_{\cdot}^{x})$ for each $\omega \notin \Delta(u) \bigcup \Delta$.
\end{itemize}

Fix a point $x_0 \in M$.  For each sequence $\{x_k\}_{k=1}^{\infty}$ with $\lim_{k \rightarrow \infty}x_k=x_0$, we
take a sequence $\{u_k\}_{k=1}^{\infty}$ and $U_0$ in $OM$, such that
$\pi(u_k)=x_k$, $\pi(U_0)=x_0$ and $\lim_{k \rightarrow \infty}u_k=u_0$ in $OM$.
Set $\tilde \Delta:=\left(\bigcup_{k=0}^{\infty}\Delta(u_k)\right)\bigcup \Delta$. For each $k$ and $\omega \notin \tilde \Delta$, $\zeta(U_k,\omega)=\zeta\left({x_k},\omega\right)$.
By the lower semi-continuity of $\zeta$,  $\zeta({x_0})\le \liminf_{k \rightarrow \infty}\zeta({x_k})$, hence $u_k \in \Xi_t(\omega)$ for each $t<\zeta({x_0})$ when $k$ is large enough. By the property (1) above, we have immediately
$$\lim_{k \rightarrow \infty} U_t(u_k, \omega)\1_{\{t<\zeta({x_k})\}}=U_t(u_0,\omega)\1_{\{t<\zeta({x_0})\}}, \;
\quad \omega \notin \tilde \Delta,\ t>0.$$
Combing this with the definition $\Theta(s,X_{\cdot},v)$ and the expression \eqref{eq2.5} of $l_m$, we see
that
\begin{equation}\label{l5-1-1}
\lim_{k \rightarrow \infty}\Theta\left(s,X_{\cdot}^{x_k},V(x_k)\right)=
\Theta\left(s,X_{\cdot}^{x_0},V(x_0)\right),\quad \ s>0.
\end{equation}

Let $h(s, X_{\cdot},V(x)):=\Big(\frac{t-2s}{t}\vee 0\Big)\cdot l_m(s,X_{\cdot})\cdot u_0(x)^{-1}V(x)$, where $u_0(\cdot)$ is a
smooth section of $OM$ with $\pi(u_0(x))=x$.  We only need to demonstrate the proof for one of the term in $I(t,X_{\cdot}^{x},v)$, for this we set
\begin{equation*}
\begin{split}
w(x)&:=
\Ee\left[f\left(X_t^x\right)\left(\int_0^{\frac{t}{2}}
\left\langle \Theta\left(s,X_{\cdot}^x,V(x)\right), \d B_s\right\rangle\right)\1_{\{t<\zeta(x)\}}\right]\\
&=:\Ee\left[f\left(X_t^x\right)\left(\int_0^{\frac{t}{2}}
\left\langle\left(h'(s)+\frac{1}{2}\ric_{U_s}h(s)\right),\d B_s\right\rangle\right)\1_{\{t<\zeta(x)\}}\right],
\end{split}
\end{equation*}
  For simplicity, we only prove the continuity for
the function $x\rightarrow w(x)$, the continuity property for the other terms in $F(x)$ could be verified similarly.

According to \eqref{eq2.4} we obtain
\begin{equation*}
\sup_{k>0}\Ee\left[\left|\int_0^{\frac{t}{2}}\left\langle
\Theta\left(s,X_{\cdot}^{x_k},V(x_k)\right),\d B_s\right\rangle\right|^4\right]<\infty.
\end{equation*}
Based on this and \eqref{l5-1-1} we have
\begin{equation*}
\lim_{k \rightarrow \infty}\Ee\left[\left|\int_0^{\frac{t}{2}}\left\langle
\Theta\left(s,X_{\cdot}^{x_k},V(x_k)\right),\d B_s\right\rangle-
\int_0^{\frac{t}{2}}\left\langle
\Theta\left(s,X_{\cdot}^{x_0},V(x_0)\right),\d B_s\right\rangle \right|^2\right]=0.
\end{equation*}
Similarly from \eqref{l5-1-1} we arrive at
\begin{equation*}
\lim_{k \rightarrow \infty}\Ee\left[\left|f(X_t^{x_k})\1_{\{t<\zeta({x_k})\}}-f(X_t^{x_0})\1_{\{t<\zeta({x_0})\}}\right|^2\right]=0.
\end{equation*}
Therefore by Cauchy-Schwartz inequality
\begin{equation*}
\begin{split}
&\lim_{k \rightarrow \infty}|w(x_k)-w(x_0)|^2\\
&\le 2\|f\|_{\infty}^2\lim_{k \rightarrow \infty}\Ee\left[\left|\int_0^{\frac{t}{2}}\left\langle
\Theta\left(s,X_{\cdot}^{x_k},V(x_k)\right),\d B_s\right\rangle-
\int_0^{\frac{t}{2}}\left\langle
\Theta\left(s,X_{\cdot}^{x_0},V(x_0)\right),\d B_s\right\rangle \right|^2\right]\\
&+ 2\sup_{k>0}\Ee\left[\left|\int_0^{\frac{t}{2}}\left\langle
\Theta\left(s,X_{\cdot}^{x_k},V(x_k)\right),\d B_s\right\rangle\right|^4\right]\cdot\lim_{k \rightarrow \infty}\Ee\left[\left|f(X_t^{x_k})\1_{\{t<\zeta({x_k})\}}-f(X_t^{x_0})\1_{\{t<\zeta({x_0})\}}\right|^2\right]\\
&=0
\end{split}
\end{equation*}
Since $\{x_k\}_{k=1}^{\infty}$ is arbitrarily chosen, $w(\cdot)$ is continuous at $x_0\in M$.
Again $x_0$ is arbitrary, so $w(\cdot)$ is continuous on $M$.
This completes the proof for the lemma.
\end{proof}
\end{appendix}

\begin{acks}[Acknowledgments]
We would like to thank Christian B\"ar and Robert Neel for  helpful comments and the referees for their valuable comments and suggestions.
\end{acks}

\begin{funding}
The research of Xin Chen is supported by the National Natural Science Foundation of China (No. 12122111).
The research of Xue-Mei Li is supported by EPSRC(Nos. EP/E058124/1, S023925/1, and EP/V026100/1).
The research of Bo Wu is supported by the National Natural Science Foundation of China (No. 12071085).
\end{funding}

\bibliographystyle{imsart-number}
\bibliography{AOP1599}

\end{document}